\documentclass{amsart}

\usepackage{cancel}
\usepackage{soul}
\usepackage[normalem]{ulem}
\usepackage[all]{xy}

\usepackage{amssymb,amsmath,graphicx,mathabx}

\usepackage{xcolor}
\usepackage{enumerate}

\newtoks\prt

\newtheorem{thm}{Theorem}[section]

\newtheorem{lemma}[thm]{Lemma}
\newtheorem{prop}[thm]{Proposition}
\newtheorem{cor}[thm]{Corollary}

\newtheorem{obs}[thm]{Observation}

\newtheorem{fact}[thm]{Fact}

\theoremstyle{definition}

\newtheorem{remark}[thm]{Remark}

\def\eqn#1$$#2$${\begin{equation}\label#1#2\end{equation}}

\def\1{\boldsymbol{1}}
\def\A{\mathcal A}
\def\B{\mathcal B}
\def\C{\mathcal C}
\def\D{\mathcal D}

\def\F{\mathcal F}

\def\ce{\mathbb C}

\def\co{\operatorname{conv}}
\def\ep{\varepsilon}
\def\K{\mathcal K}
\def\en{\mathbb N}
\def\er{\mathbb R}

\def\ef{\mathbb F}

\def\Im{\operatorname{Im}}

\def \f{\boldsymbol{f}}
\def \g{\boldsymbol{g}}
\def \h{\boldsymbol{h}}
\def \uu{\boldsymbol{u}}

\def\ov{\overline}
\def \Ch {\operatorname{Ch}}

\def \ext {\operatorname{ext}}

\def\span{\operatorname{span}}
\def\id{\operatorname{id}}

\def \reg {\partial _{\kern1pt\text{reg}}}

\def\iff{\Longleftrightarrow}

\def\Co{\operatorname{Con}}

\def\di{\,\mbox{\rm d}}

\newcommand{\norm}[1]{\left\|#1\right\|}

\renewcommand{\Re}{\operatorname{Re}}

\newcommand{\abs}[1]{\left|#1\right|}
\newcommand{\setsep}{;\,}

\numberwithin{equation}{section}

\definecolor{green}{rgb}{0,0.5,0}
\title {Transference of measures via disintegration}
\author{Ond\v{r}ej F.K. Kalenda and Ji\v{r}\'{\i} Spurn\'y}

\address{Ond\v{r}ej F.K. Kalenda\\
Charles University\\
Faculty of Mathematics and Physics\\
Department of Mathematical Analysis \\
Sokolovsk\'{a} 83, 186 \ 75\\Praha 8, Czech Republic}
\email{kalenda@karlin.mff.cuni.cz}

\address{Ji\v{r}\'{\i} Spurn\'y\\
Charles University\\
Faculty of Mathematics and Physics\\
Department of Mathematical Analysis \\
Sokolovsk\'{a} 83, 186 \ 75\\Praha 8, Czech Republic}
\email{spurny@karlin.mff.cuni.cz}

\keywords{vector-valued functions and measures; disintegration of measures; ordering of measures; minimal and maximal measures}

\subjclass[2010]{46G10; 46A55;  28A35; 28A50}

\thanks{Our research was partially supported by the Research grant GA\v{C}R 23-04776S}

\begin{document}
\begin{abstract}
    Given a compact space $K$ and a Banach space $E$ we study the structure of positive measures
    on the product space $K\times B_{E^*}$ representing functionals on $C(K,E)$, the space of $E$-valued continuous functions on $K$. Using the technique of disintegration we provide an alternative approach to the procedure of transference of measures introduced by Batty (1990). 
    This enables us to substantially strengthen some of his results, to discover a rich order
    structure on these measures, to identify maximal and minimal elements and to relate them to the classical Choquet order.
\end{abstract}

\maketitle

\section{Introduction}

The classical Riesz representation theorem provides a bijective isometric correspondence between continuous linear functionals on $C(K)$, the space of (real- or complex-valued) continuous functions on a Hausdorff compact space $K$, and $M(K)$, the space of (signed or complex) Radon measures on $K$.
Therefore, given a subspace $H\subset C(K)$, the mentioned Riesz theorem together with the Hahn-Banach extension theorem yield that any continuous linear functional on $H$ may be represented by a Radon measure on $K$ with the same norm. Such a representing measure need not be unique, hence it makes sense to compare the representing measures and to investigate their structure. This is the basic content of the Choquet theory. 

 In the classical case $K$ is a compact convex set and $H=A_c(K)$ is the space of all affine continuous functions on $K$. If $K$ is metrizable, the classical Choquet theorem yields that any continuous linear functional on $A_c(K)$ is represented by a  measure $\mu$ with the same norm that is carried by the set $\ext K$ of extreme points of $K$. For non-metrizable $K$ the question is more subtle, the Choquet ordering naturally arises and one gets a representing measure $\mu$ ``almost carried'' by $\ext K$  in the sense that $\abs{\mu}(K\setminus B)=0$ for each Baire set $B\supset \ext K$. This is summarized in the famous Choquet-Bishop-de-Leeuw theorem, see \cite[Chapter I, \S 4]{alfsen}. The question of uniqueness in this context leads to the theory of Choquet simplices (see, e.g., \cite[Section II.3]{alfsen}).

A generalization of this representation theory is encompassed by the notion of a function space $H$, which is a subspace of $C(K,\er)$ containing constant functions and separating points of $K$ (see e.g. \cite[Chapter 6]{phelps-choquet} or \cite[Chapter 3]{lmns}). In this context, the role of the set of extreme points is played by the so-called Choquet boundary $\Ch_H K$ of $H$ (this is the set of those points $x\in K$ such that the point evaluation at $x$ is an extreme point of the dual unit ball $B_{H^*}$ of $H$). Again we may construct for a functional $s\in H^*$  a representing measure with the same norm as $s$ which is ``almost carried'' by $\Ch_H K$. The question of uniqueness leads to the theory of simplicial function spaces (see, e.g., \cite[Chapter 6]{lmns}).

The next step was a generalization to the complex case which was addressed in \cite{hustad71,hirsberg72,fuhr-phelps,phelps-complex}. It turns out that the representation
theorem holds in the same form, but the question of uniqueness is more subtle than in the real case. 

A further generalization deals with vector-valued function spaces, i.e., subspaces $H\subset C(K,E)$ for some compact $K$ and Banach space $E$.  A satisfactory theory of integral representation in this case was presented by P. Saab and M. Talagrand in a series of papers \cite{saab-aeq,saab-canad,saab-tal}. Their approach was further improved by W.~Roth and C.J.K. Batty in \cite{roth-london,roth-convex,roth-kniha,batty-vector}. In the vector-valued case there are two basic approaches to the representation -- via vector measures on $K$, using a generalization of the Riesz theorem saying that the dual to $C(K,E)$ is canonically isometric to $M(K,E^*)$, or via scalar measures on $K\times B_{E^*}$ using the canonical isometric inclusion $T:C(K,E)\to C(K\times B_{E^*})$ defined by
$$T\f(t,x^*)=x^*(\f(t)), \quad (t,x^*)\in K\times B_{E^*}, \f\in C(K,E).$$
These two approaches are closely related. In fact, the construction of representing vector measures in the quoted papers is done via the scalar measures -- at first a suitable representing measure on $K\times B_{E^*}$ is constructed and then the respective vector measure is obtained by application of the `Hustad mapping' (which is inspired by \cite{hustad71} and can be viewed as an interpretation of the dual operator to the above-defined inclusion $T$). This procedure was used in \cite{hensgen} to provide a simple proof of the representation of the dual to $C(K,E)$ (see Section~\ref{ssec:dualC(K,E)}) and was substantially elaborated by Batty in \cite{batty-vector} using the technique of `transference of measures' which provides a canonical way how to assign to each $\mu\in M(K,E^*)$ a 
positive measure on $K\times B_{E^*}$ (with the same norm and whose image under $T^*$ is $\mu$).

Our aim is to further investigate the vector-valued integral representation theory, in particular how to grasp the notion of uniqueness of representing measures. To this end we investigate in more detail the above-mentioned procedure discovered by Batty. Since the procedure itself does not depend on the choice of $H$, we restrict ourselves to the case $H=C(K,E)$. 

In this case the representation by vector-valued measures trivially reduces to a known theorem saying that each functional on $C(K,E)$ is represented by a unique vector measure (see Section~\ref{ssec:dualC(K,E)}). However, the structure of representing positive measures on $K\times B_{E^*}$ is nontrivial. In accordance with Section~\ref{s:ordering} below let us denote, for a given $\mu\in M(K,E^*)$, by $N(\mu)$ the set of all positive measures on $K\times B_{E^*}$ with norm $\norm{\mu}$ and representing the same functional as $\mu$. The main tool we use to investigate the structure of $N(\mu)$ is the technique of disintegration of measures on product spaces (see Section~\ref{ssec:disint}). With the help of this technique, we obtain (among others) the following results:

For a vector measure $\mu\in M(K,E^*)$, there is a weak$^*$ Radon-Nikod\'ym derivative  $h$ of $\mu$ with respect to $\abs{\mu}$ (see Proposition~\ref{P:hustotaT*nu} and the rest of Section~\ref{s:tovector}). This enables us to give an alternative proof of \cite[Proposition 3.3]{batty-vector} (which yields a canonical selection mapping of the assignment $\mu\mapsto N(\mu)$ denoted by $K$ in \cite{batty-vector} and by $W$ in the present paper) and to provide a formula for the operator $W$ (see Corollary~\ref{c:battytrans}). We point out that such a formula is given in \cite[Proposition 2.2]{batty-vector} but only under the very strong assumption of the existence of the Bocher-Radon-Nikod\'ym derivative of $\mu$ with respect to $\abs{\mu}$, while our approach provides the formula in full generality.

A further application is Theorem~\ref{t:battymain} which shows that $N(\mu)$ is a singleton for each $\mu$ (i.e., the scalar representing measures are unique) if and only if $E^*$ is strictly convex. This is the optimal version of a result from \cite[p. 540]{batty-vector}, where the uniqueness is proved under much stronger assumptions (in particular, $E$ is required to be separable and reflexive there).

 We also analyze in detail a partial order $\prec_{\D}$ on $N(\mu)$ introduced in \cite{batty-vector}. We use the method of disintegration to relate it with the Choquet order on $M_1(B_{E^*})$ (see Theorem~\ref{t:precDbodove}). The set $N(\mu)$ contains the $\prec_{\D}$-largest element
 (it coincides with the value $W\mu$ of the above-mentioned operator $W$, see Section~\ref{ss:ordering-basic}). We further characterize $\prec_{\D}$-minimal elements of $N(\mu)$ using maximal measures on $M_1(B_{E^*})$ (see Theorem~\ref{t:minimalbodove}). Finally, we show that each $N(\mu)$ contains a unique $\prec_{\D}$-minimal measures  if and only if the dual unit ball $B_{E^*}$ is a simplexoid (see Theorem~\ref{t:simplexoidminimal}). 

We are convinced that the results and techniques from the present paper will be useful  to investigate the integral representation in case  $H\subsetneq C(K,E)$, which we plan to continue elsewhere.
This general case will, however, require some additional effort, in particular, because the Choquet boundary and $H$-boundary measures should be considered (in case $H=C(K,E)$, the Choquet boundary is whole $K$ and all measures are $H$-boundary). Thus the representation theorem is more involved (see \cite{saab-tal}) and possible analogues of relations $\prec_{\B}$ and $\prec_{\D}$ are more complicated. 
Moreover, some special properties of $C(K,E)$ (Urysohn lemma, functions of the form $f\cdot x$ etc.) are not availaible for general $H$.
 
\section{Preliminaries}

In this section we collect some notation and auxiliary results (mostly known) which will be used in the sequel.
We divide this section to several subsections.

\subsection{Real and complex Banach spaces} 

The classical Choquet theory deals with real spaces, the complex case requires some additional effort as recalled in the introduction. In \cite{batty-vector}, which is one of our main references, the real spaces are considered as well. However, our results hold for real and complex spaces equally.

We will denote by $\ef$ the respective field ($\er$ or $\ce$). Moreover, we will repeatedly use without commenting it the following standard facts on complex Banach spaces:

If $E$ is a complex Banach space and $E_R$ is its real version (i.e., the same space considered over $\er$), then we have the following identifications:

\begin{itemize}
    \item The assignment $x^*\mapsto \Re x^*$ is a real-linear isometry of $E^*$ onto $(E_R)^*$.
    \item Conversely, if $y^*\in (E_R)^*$, then the formula $x^*(x)=y^*(x)-iy^*(ix)$, $x\in E$, defines an element of $E^*$ with $y^*=\Re x^*$ (and $\norm{x^*}=\norm{y^*}$).
\end{itemize}

\subsection{Classical Choquet theory}\label{ss:classical}

This section recalls classical notions of the Choquet theory of compact convex sets. Assume that $X$ is a compact convex set in a locally convex Hausdorff space. Then for each $\mu\in M_1(X)$ (the set of all Radon probability measures on $X$) there exists a unique point $x=r(\mu)\in X$ (called the \emph{barycenter of $\mu$}) satisfying $\int_X f\di\mu=f(x)$ for each affine continuous function $f\colon X\to \er$.
If $x=r(\mu)$, we say that \emph{$\mu$ represents $x$}.

The \emph{Choquet order $\prec$} on the cone $M_+(X)$ of all Radon positive measures on $X$ is defined as $\mu\prec\nu$ for $\mu,\nu\in M_+(X)$ if and only if $\int k\di\mu\le \int k\di\nu$ for each $k:X\to\er$ convex and continuous. A \emph{maximal} measure then means a measure maximal in the ordering $\prec$.

The maximality of a measure $\mu\in M_+(X)$ can be characterized by means of envelopes. Recall that, given a bounded real-valued function $f$ on $X$, its \emph{upper} and \emph{lower envelopes} are defined as
\begin{equation}
\label{eq:obalky}  
\begin{aligned}
f^*&=\inf\{h\setsep h\ge f, h\in C(X,\er)\text{ affine}\}
\quad\text{and}\\
f_*&=\sup\{h\setsep h\le f, h\in C(X,\er)\text{ affine}\}.
\end{aligned}
\end{equation}
Then $\mu\in M_+(X)$ is maximal if and only if $\int f\di\mu=\int f^*\di\mu$ for each convex continuous function $f:X\to\er$ (this result is due to Mokobodzki, see \cite[Proposition I.4.5.]{alfsen}).

If $X$ is metrizable and $\ext X$ stands for the set of all extreme points of $X$, then  $\ext X$ is a $G_\delta$-subset of $X$ and $\mu\in M_+(X)$ is maximal if and only if $\mu(X\setminus \ext X)=0$,  see \cite[p. 35]{alfsen}. 

It is an easy consequence of Zorn's lemma that for any $\mu\in M_+(X)$ there is a maximal measure $\nu\in M_+(X)$ such that $\mu\prec\nu$. In case $X$ is metrizable, there is a witnessing Borel assignment provided by the following lemma. 

\begin{lemma}\label{L:selekce max}
   Let $X$ be a metrizable compact convex set. Then there exists a Borel measurable mapping
   $\Psi:M_1(X)\to M_1(X)$ such that $\nu\prec\Psi(\nu)$ and $\Psi(\nu)$ is a maximal measure for each $\nu\in M_1(X)$.
\end{lemma}

\begin{proof}
    By \cite[Theorem 11.41]{lmns} there is a Borel measurable mapping $m:X\to M_1(X)$ such that 
    $m(x)$ is a maximal measure representing $x$ for each $x\in X$. Fix $\mu\in M_1(X)$. Let us define a functional $\psi_\mu$ on $C(X)$ by
    $$\psi_\mu(f)=\int_X\left(\int_X f(y)\di m(x)(y)\right)\di\mu(x), \quad f\in C(X).$$
    Note that, given $f\in C(X)$, the function $x\mapsto \int f\di m(x)$ is Borel measurable and bounded by $\norm{f}$, so $\psi_\mu$ is a well-defined linear functional of norm at most $\norm{\mu}$. Let $\Psi(\mu)$ be the measure representing $\psi_\mu$.

    To observe that $\Psi$ is a Borel mapping, it is enough to show that $\mu\mapsto \psi_\mu(f)$ is Borel measurable for each $f\in C(X)$. We already know that the function $x\mapsto \int f\di m(x)$ is Borel measurable, so it is a Baire function, thus $\mu\mapsto\psi_\mu$ is also a Baire function.

    If $f\in C(X,\er)$ is convex, then
    $$\int f\di \Psi(\mu)=\psi_\mu(f)=\int_X\left(\int f(y)\di m(x)(y)\right)\di\mu(x)\ge \int_X f(x)\di\mu(x),$$
    where we used that $f$ is convex and $m(x)$ represents $x$ for each $x\in X$. We deduce that $\mu\prec\Psi(\mu)$.

    Finally, let us show $\mu$ is maximal, i.e, it is carried by $\ext X$. By construction we have
    $$\int f\di\Psi(\mu)=\int_X\left(\int f(y)\di m(x)(y)\right)\di\mu(x), \quad f\in C(X).$$
    By the Lebesgue dominated convergence theorem this equality extends to bounded Baire functions on $X$, so to bounded Borel functions on $X$. In particular, applying to the characteristic function of $X\setminus\ext X$ we deduce
    $$\Psi(\mu)(X\setminus\ext X)=\int_X m(x)(X\setminus\ext X)\di\mu(x)=0.$$
\end{proof}

\subsection{Integration with respect to measures with values in a dual Banach space}

Let $(\Omega,\Sigma)$ be a measurable space, let $E$ be a (real or complex) Banach space. If $\mu$ is an $E^*$-valued $\sigma$-additive measure on $(\Omega,\Sigma)$, we denote by $\abs{\mu}$ its (absolute) variation (see \cite[Definition 4 on p. 2]{diesteluhl}) and we set $\norm{\mu}=\abs{\mu}(\Omega)$, the total variation of $\mu$. If $\norm{\mu}<\infty$, $\mu$ is said to have \emph{bounded variation}.
Moreover, $\mu$ is called \emph{regular} if its variation $\abs{\mu}$ is regular.

Assume that $\mu$ is an $E^*$-valued $\sigma$-additive measure on $(\Omega,\Sigma)$ with bounded variation. If $x\in E$, the formula
$$\mu_x(A)=\mu(A)(x),\quad A\in\Sigma,$$
defines a scalar-valued $\sigma$-additive measure on $(\Omega,\Sigma)$ with bounded variation (more precisely, we have $\norm{\mu_x}\le \norm{\mu}\cdot\norm{x}$.

Assume that $\uu=\sum_{j=1}^n \chi_{A_j}\cdot x_j$ is a simple measurable function (where $x_1,\dots, x_n\in E$ and  $A_1,\dots, A_n$ are pairwise disjoint elements of $\Sigma$). Then we define
$$\int \uu\di\mu=\sum_{j=1}^n \mu(A_j)(x_j).$$
It is easy to check that the mapping $\uu\mapsto \int \uu\di\mu$ is linear (from the space of simple measurable functions to $\ef$). Moreover,
$$\abs{\int \uu\di\mu}\le\sum_{j=1}^n \abs{\mu(A_j)(x_j)}\le \sum_{j=1}^n \norm{\mu(A_j)}\norm{x_j}\le \norm{\mu}\cdot\norm{\uu}_\infty,$$
hence the integral may be uniquely extended to those functions $\f:\Omega\to E$ which may be uniformly approximated by simple measurable functions. In particular, if $f:\Omega\to \ef$ is a bounded measurable function and $x\in E$, the function $f\cdot x$ is $\mu$-integrable and we have
$$\int f\cdot x \di\mu =\int f\di\mu_x.$$

Further, if $K$ is a compact space, then any continuous function from $K$ to $E$ may be uniformly approximated  by simple Borel measurable functions, and thus we may define
$\int \f\di\mu$ whenever $\f:K\to E$ is continuous and $\mu$ is an $E^*$-valued Borel measure on $K$ with bounded  variation. In this case we have
$$\abs{\int \f\di\mu}\le \norm{f}_\infty\norm{\mu}.$$

An important special type of vector measures are those of the form $\ep_t\otimes x^*$ where $t\in K$ and $x^*\in E^*$. Such measures act as follows
$$\begin{aligned}
(\ep_t\otimes x^*)(B)&=\begin{cases} x^*, & t\in B,\\ 0, &t\notin B,    
\end{cases} \quad B\subset K\mbox{ Borel}, \\ \int \f\di(\ep_t\otimes x^*)&=x^*(\f(t)),\ \f\in C(K,E).\end{aligned}$$

\subsection{Representation of the dual to $C(K,E)$}\label{ssec:dualC(K,E)}  

The integral from the previous section may be used to provide a representation of the dual to the space of vector-valued continuous functions. Let us fix the relevant notation. Let $K$ be a compact space and let $E$ be a (real or complex) Banach space. By $C(K,E)$ we denote the Banach space of $E$-valued continuous functions on $K$ with the supremum norm. By $M(K,E^*)$ we denote the space of all regular $E^*$-valued Borel measures on $K$ with bounded variation, 
equipped with the total variation norm. Then $M(K,E^*)$ is canonically isometric to the dual of $C(K,E)$. Let us explain it a bit.

It follows from the previous section that any $\mu\in M(K,E^*)$ induces a continuous linear functional on $C(K,E)$ of norm at most $\norm{\mu}$ by
$$\f\mapsto \int \f\di\mu.$$
Conversely, assume that $\varphi\in C(K,E)^*$ is given. For each $x\in E$ define
$$\varphi_x(f)=\varphi(f\cdot x),\quad f\in C(K,\ef).$$
Then $\varphi_x\in C(K,\ef)^*$ and $\norm{\varphi_x}\le \norm{\varphi}\norm{x}$. By the Riesz representation theorem there is a measure $\mu_x\in M(K,\ef)$ with $\norm{\mu_x}\le\norm{\varphi}\norm{x}$ representing $\varphi_x$. Moreover, since the assignment $x\mapsto \varphi_x$ is linear and continuous
(of norm at most $\norm{\varphi}$), the mapping $x\mapsto \mu_x$ is a bounded linear operator
(from $E$ to $M(K,\ef)$). For a Borel set $B\subset K$ define 
$$\mu(B)(x)=\mu_x(B), \quad x\in E.$$
Then $\mu$ is obviously a finitely additive mapping from the Borel $\sigma$-algebra to $E^*$. Moreover, $\mu$ is a regular $\sigma$-additive measure with bounded variation representing $\varphi$ and satisfying $\norm{\mu}\le\norm{\varphi}$. An easy proof of this fact is provided in \cite{hensgen}. Since we will repeatedly use the related procedure, we briefly recall the argument.

Let $T:C(K,E)\to C(K\times B_{E^*})$ be defined by
$$T\f(t,x^*)=x^*(\f(t)), \quad (t,x^*)\in K\times B_{E^*}, \f\in C(K,E).$$
Then $T$ is a linear isometric injection. By the Riesz representation theorem, the space $C(K\times B_{E^*})^*$ is canonically isometric to $M(K\times B_{E^*})$, so the dual mapping $T^*$ may be considered as a mapping
$T^*:M(K\times B_{E^*})\to C(K,E)^*$. So, continuing from the previous paragraph, there is $\nu\in M(K\times B_{E^*})$ such that $\norm{\nu}=\norm{\varphi}$ and $T^*\nu=\varphi$. By the definition of the dual mapping we deduce
$$\varphi(\f)=\int x^*(\f(t))\di\nu(t,x^*),\quad \f\in C(K,E).$$
In particular, for each $x\in E$ and $f\in C(K)$ we have
$$\int_K f\di \mu_x=\varphi(f\cdot x)=\int f(t) x^*(x)\di \nu(t,x^*),$$
so
$$\mu_x(A)= \int_{A\times B_{E^*}} x^*(x)\di \nu(t,x^*),\quad A\subset K\mbox{ Borel}.$$
It follows that
$$\norm{\mu(A)}\le \abs{\nu}(A\times B_{E^*}), \quad A\subset K\mbox{ Borel}.$$
It now easily follows that $\mu$ is $\sigma$-additive, regular and $\norm{\mu}\le \norm{\nu}=\norm{\varphi}$. Moreover,
$$\varphi(f\cdot x)=\int f\cdot x\di\mu, \quad f\in C(K), x\in E.$$
Since these functions are linearly dense in $C(K,E)$, we conclude that $\mu$ represents $\varphi$ in the 
sense of the previous section.

Finally, we may  interpret $T^*$ as a mapping $T^*:M(K\times B_{E^*})\to M(K,E^*)$. By the construction we
have
\begin{equation}\label{eq:hustad}    
T^*\nu(A)(x)=\int_{A\times B_{E^*}} x^*(x)\di \nu(t,x^*),\quad A\subset K\mbox{ Borel}, x\in E.\end{equation}
Hence $T^*$ in this representation coincides with the Hustad mapping used in \cite{batty-vector} and elsewhere.

\subsection{Batty's correspondences}\label{s:dualita}

In this section we briefly recall the canonical correspondences established in \cite[Section 2]{batty-vector} and used then in the procedure of the `transference of measures'. One of the basic tools for these correspondences is the following lemma which is repeatedly implicitly used in the proofs in \cite{batty-vector}.

\begin{lemma}\label{L:HB}
     Let $X$ be a real locally convex space and let $p:X\to\er$ be a lower semicontinuous sublinear functional. Then $$p(x)=\sup\{f(x)\setsep f\in X^*, f\le p\}, \quad x\in X.$$
     The sup cannot be replaced by max.
\end{lemma}

The positive part of this lemma is a (rather standard but non-trivial) consequence of the Hahn-Banach separation theorem. Since we have not found any reference for this result (except for a far more general version \cite[Theorem 2.11]{hypolinear} with a complicated proof) we decided to present here a simple elementary proof, essentially following the argument used in \cite[p. 534]{batty-vector} in a special case.

\begin{proof} Let us start by the negative part. A possible counterexample is given in \cite[Example 2.10]{hypolinear} using a non-complete inner product space. Another possibility is to take $X=(Y^*,w^*)$ where $Y$ is any nonreflexive Banach space, $p(y^*)=\norm{y^*}$ for $y^*\in Y^*$ and a functional $y_0^*\in Y^*$ not attaining the norm.

To prove the positive part fix any $x\in X$ and any $c<p(x)$. Let 
$$A=\{(y,t)\in X\times \er\setsep t\ge p(y)\}.$$
Then $A$ is a closed convex set and $(x,c)\notin A$. Applying the Hahn-Banach separation theorem in $X\times \er$, we find $f\in X^*$ and $d\in\er$ such that
$$f(x)+cd<\inf\{ f(y)+dt\setsep (y,t)\in A\}.$$
Necessarily $d\ge0$, otherwise the right-hand side would be $-\infty$.  It follows that
$$f(x)+cd<\inf\{ f(y)+dp(y)\setsep y\in X\}.$$
Since on the right-hand side we may choose $y=x$, we deduce $d>0$. So, without loss of generality $d=1$. I.e., 
$$\begin{aligned}
f(x)+c&<\inf\{ f(y)+p(y)\setsep y\in X\}=\inf\{ f(ty)+p(ty)\setsep y\in X,t\ge0\}\\&= \inf\{ t(f(y)+p(y))\setsep y\in X,t\ge0\}.\end{aligned}$$
It follows that the right-hand side is either $0$ or $-\infty$. But the second possibility cannot take place, so it must be $0$. It follows $f(x)+c<0$, so $-f(x)>c$. Moreover, $f(y)+p(y)\ge0$ for $y\in X$, so $-f\le p$.
 This completes the proof.  
\end{proof}

We continue by an abstract version of some of the correspondences from \cite{batty-vector}.

\begin{lemma}\label{L:abstract-batty}
  Let $X$ be a (real or complex) Banach space.
  \begin{enumerate}[$(a)$]
      \item  If $U\subset X$ is a nonempty closed convex bounded set, we set
$$p_U(x^*)=\inf \{\Re x^*(x)\setsep x\in U\},\quad x^*\in X^*.$$ 
Then $p_U$ is a weak$^*$ upper semicontinous superlinear functional.
\item If $p:X^*\to\er$ is a weak$^*$ upper semicontinuous superlinear functional, we set 
$$U_p=\{x\in X\setsep \Re x^*(x)\ge p(x^*)\mbox{ for }x^*\in X^*\}.$$
Then $U_p$ is a nonempty closed convex bounded set.
\item If $U\subset X$ is a nonempty closed convex bounded set, then $U_{p_U}=U$.
\item If $p:X^*\to\er$ is a weak$^*$ upper semicontinuous superlinear functional, then $p_{U_p}=p$.
  \end{enumerate}
\end{lemma}

\begin{proof}
    Assertion $(a)$ is obvious. Let us continue by proving $(b)$. It is clear that $U_p$ is closed and convex. Further, $U_p\ne\emptyset$ by Lemma~\ref{L:HB} applied to $-p$. To prove it is bounded, observe that for each $x\in U_p$ and $x^*\in X^*$ we have
$$p(x^*)\le \Re x^*(x) = -\Re (-x^*)(x)\le -p(-x^*).$$
In case $\ef=\ce$ we also have  $\Im x^*(x)=\Re (-ix^*)(x)$. So, in any case the set
$\{x^*(x)\setsep x\in U_p\}$ is bounded for each $x^*\in X^*$. By the uniform boundedness principle we deduce that $U_p$ is bounded.

$(c)$:  Obviously $U\subset U_{p_U}$. Conversely, if $x\notin U$, by the separation theorem there is $x^*\in X^*$ such that
$$\Re x^*(x) < \inf \{ \Re x^*(y)\setsep y\in U\}=p_U(x^*),$$
so $x\notin U_{p_U}$.

Assertion $(d)$ follows from Lemma~\ref{L:HB} applied to $-p$.
\end{proof}

Now we pass to the correspondences related to $C(K,E)$. Recall that $M(K,E^*)$ is canonically isometric to the dual of $C(K,E)$, so it is equipped with the related weak$^*$ topology. 
We consider the following four families:
    $$\begin{aligned}
      \A&=\{U\subset C(K,E)\setsep U\mbox{ is nonempty closed bounded and $C(K)$-convex}\},\\
      \B&=\{p:M(K,E^*)\to\er\setsep p\mbox{ is weak$^*$ upper semicontinuous and superlinear,}
      \\&\qquad\qquad\qquad\qquad
      p(\mu_1+\mu_2)=p(\mu_1)+p(\mu_2)\mbox{ whenever }\mu_1\perp\mu_2\},\\
      \C&=\{\psi:K\to 2^E\setsep \psi\mbox{ is lower semicontinuous, bounded }\\&\qquad\qquad\qquad\qquad\mbox{and has nonempty closed convex values}\}, \\
      \D&=\{f:K\times E^*\to\er\setsep f|_{K\times B_{E^*}}\mbox{ is upper semicontinuous and bounded,}\\&\qquad\qquad\qquad\qquad f(t,\cdot)\mbox{ is superlinear for each }t\in K\}.
    \end{aligned}$$
Note that $U\subset C(K,E)$ is $C(K)$-convex if $h\f+(1-h)\g\in U$ whenever $\f,\g\in U$
and $h\in C(K)$ satisfies $0\le h\le 1$.

\begin{prop}
    Let $K$ be a compact space and let $E$ be a Banach space. The above-defined families $\A,\B,\C,\D$ are in compatible bijective correspondences:
    $$ \xymatrix@C+1pc{ \B  \ar@<.5ex>[r]^{p\mapsto U_p} & \A \ar@<.5ex>[l]^{U\mapsto p_U} \ar@<.5ex>[r]^{U\mapsto\psi_U}& \C \ar@<.5ex>[l]^{\psi\mapsto U_\psi} \ar@<.5ex>[r]^{\psi\mapsto f_\psi} & \D \ar@<.5ex>[l]^{f\mapsto \psi_f}}
    $$
   The correspondences between $\A$ and $\B$ are given by the formulas from Lemma~\ref{L:abstract-batty}, the remaining ones are given by the formulas
   $$\begin{aligned}
       \psi_U(t)&=\overline{\{\f(t)\setsep \f\in U\}}=\{\f(t)\setsep \f\in U\}, \quad t\in K, U\in\A,
       \\ U_\psi&=\mbox{all continuous selections from }\psi,\quad \psi\in\C,    \\
       f_\psi(t,x^*)&=\inf \{\Re x^*(x)\setsep x\in \psi(t)\},\quad (t,x^*)\in K\times E^*,\psi\in \C,\\ \psi_f(t)&=\{ x\in E\setsep \Re x^* \ge f(t,\cdot)\},\quad t\in K, f\in\D.
   \end{aligned}$$
\end{prop}

This proposition is proved in \cite[Theorem 2.1]{batty-vector}. Let us briefly comment it. The proof of the correspondence between $\A$ and $\B$ is based on the fact that the abstract correspondence from Lemma~\ref{L:abstract-batty} maps $\A$ into $\B$ and vice versa. The proof of the correspondence between $\A$ and $\C$ uses, among others, Michael's selection theorem. Finally, in the proof of the correspondence between $\C$ and $\D$ a uniform version of Lemma~\ref{L:HB} is used to show that $\D$ is mapped into $\C$ and then Lemma~\ref{L:abstract-batty} is used for any fixed $t\in K$.

The main application of the above correspondences is the resulting correspondence between $\B$ and $\D$ (see \cite[p. 535]{batty-vector}):

\begin{cor}\label{cor:f-p}
    The resulting correspondence between $\B$ and $\D$ is provided by the formulas
    $$\begin{aligned}
       f_p(t,x^*)&=p(\ep_t\otimes x^*), \quad (t,x^*)\in K\times E^*, p\in \B,\\
       p_f(\mu)&=\inf \Big\{ \Re \int\g\di\mu\setsep \g\in C(K,E), \Re x^*(\g(t))\ge f(t,x^*)\\&\qquad\qquad\qquad\qquad\mbox{ for }(t,x^*)\in K\times E^*\Big\}, \quad \mu\in M(K,E^*), f\in\D.
    \end{aligned}$$
\end{cor}

Functions from $\D$ are determined by their restrictions to $K\times B_{E^*}$, therefore we will often identify $f\in \D$ with its restriction.
Note that both $\B$ and $\D$ are convex cones. Let us look at $\B\cap(-\B)$ and $\D\cap(-\D)$. We have the following:

\begin{prop}\label{P:BD-BD} \
    \begin{enumerate}[$(a)$]
        \item $\B\cap(-\B)=\{\mu\mapsto \Re\int\f\di\mu\setsep \f\in C(K,E)\}$, so $\B\cap(-\B)$ is in a canonical real-linear bijective correspondence with $C(K,E)$;
        \item $\D\cap(-\D)=\{\Re T\f\setsep \f\in C(K,E)\}$;
        \item the correspondence $p\mapsto f_p$ restricted to $\B\cap(-\B)$ coincides with the operator $\f\mapsto \Re T\f$.
    \end{enumerate}
\end{prop}

\begin{proof}
    $(a)$: By definition, $\B\cap(-\B)$ consists of real-valued real-linear weak$^*$ continuous functionals on $M(K,E^*)$. Thus the assertion follows.

    $(b)$: Elements of $\D\cap(-\D)$ are continuous on $K\times B_{E^*}$ and real-linear in the second variable. Thus the assertion follows.

    $(c)$: This follows from $(a)$ and Corollary~\ref{cor:f-p}. Indeed, assume $p(\mu)=\Re\int\f\di\mu$ for some $\f\in C(K,E)$. Then
    $$f_p(t,x^*)=p(\ep_t\otimes x^*)=\Re x^*(\f(t))=\Re T\f(t,x^*).$$
\end{proof}

\subsection{Disintegration of complex measures on compact spaces}\label{ssec:disint} 

In this section we include basic results on disintegration of measures on products of compact spaces. 
Our basic source is \cite[Chapter 452]{fremlin4}. Usually the disintegration is applied to positive measures. We start by a lemma showing that this method may be easily adapted to complex measures.

\begin{lemma}\label{L:dezintegrace}
    Let $K$ and $L$ be two compact Hausdorff spaces and let $\nu$ be a complex Radon measure on $K\times L$. Denote by $\sigma$ the projection of the absolute variation $\abs{\nu}$ to $K$. Then there is an indexed family $(\nu_t)_{t\in K}$ of complex Radon measures on $L$ such that the following conditions are satisfied.
    \begin{enumerate}[$(i)$]
        \item $\norm{\nu_t}=1$ for each $t\in K$.
        \item For each continuous function $f:K\times L\to \ce$  we have
        $$\int_{K\times L} f\di\nu=\int_K\left(\int_L f(t,z)\di\nu_t(z)\right)\di\sigma(t).$$
        In particular, for any such $f$ the function
        $$t\mapsto \int_L f(t,z)\di\nu_t(z)$$
        is $\sigma$-measurable.
        \item If $A\subset K$ and $B\subset L$ are Borel sets, then
        $$\nu(A\times B)=\int_A \nu_t(B) \di\sigma(t).$$
        In particular, the mapping $t\mapsto \nu_t(B)$ is $\sigma$-measurable whenever $B\subset L$ is Borel.
    \end{enumerate}
    Moreover, if $\nu$ is positive, then all $\nu_t$ may be chosen to be probability measures.
   If $\nu$ is real-valued, all $\nu_t$ may be chosen to be real-valued.
\end{lemma}

\begin{proof}
    Denote by $\lambda$ the projection of $\abs{\nu}$ to $L$. Then both $\sigma$ and $\lambda$ are Radon measures. Let $\Sigma$ and $\Upsilon$ denote the $\sigma$-algebras of the $\sigma$-measurable and $\lambda$-measurable sets, respectively. Let $\Sigma\otimes\Upsilon$ denote the product $\sigma$-algebra and let 
    $\nu^\prime$ denote the restriction of $\abs{\nu}$ to $\Sigma\otimes\Upsilon$. We apply \cite[Theorem 452M]{fremlin4} to $\nu^\prime$ and get an indexed family $(\nu^0_t)_{t\in K}$ of Radon probabilities on $L$ such that
    \begin{equation}\label{eq:dis}\begin{aligned}
    \int_{K\times L} f\di\abs{\nu}=&\int_K\left(\int_L f(t,z)\di\nu^0_t(z)\right)\di\sigma(t) \\&\mbox{for each }f:K\times L\to \ce\mbox{ bounded $\Sigma\otimes\Upsilon$-measurable}.     \end{aligned}
    \end{equation}

 Let $h$ be the Radon-Nikod\'ym density of $\nu|_{\Sigma\otimes\Upsilon}$ with respect to $\nu^\prime$. Then $h$ is $\Sigma\otimes\Upsilon$-measurable and without loss of generality $\abs{h}=1$ everywhere on $K\times L$. Given $t\in K$, let $\nu_t$ be the measure defined by
    $$\di\nu_t= h(t,\cdot)\di\nu^0_t.$$
Then $\norm{\nu_t}=1$. If $\nu\ge0$, then we may take $h=1$, so $\nu_t=\nu^0_t$ is a probability measure. If $\nu$ is real-valued, $h$ may attain only real values ($1$ and $-1$), hence $\nu_t$ is also real-valued.
 Moreover, if $f:K\times L\to \ce$ is bounded  and $\Sigma\otimes\Upsilon$-measurable, then
    $$\begin{aligned}
   \int f\di\nu&=\int fh\di\abs{\nu}=\int_K\left(\int_L f(t,z)h(t,z)\di\nu^0_t(z)\right)\di\sigma(t)\\&=
    \int_K\left(\int_L f(t,z)\di\nu_t(z)\right)\di\sigma(t), \end{aligned}$$
    where we applied \eqref{eq:dis} to $fh$.

    Then $(iii)$ clearly holds. To prove $(ii)$ it remains to observe that continuous functions are $\Sigma\otimes\Upsilon$-measurable. This is clear for functions of the form
    $$(t,z)\mapsto f(t)g(z)\mbox{ where }f\in C(K), g\in C(L).$$
    By the Stone-Weierstrass theorem such functions are linearly dense in $C(K\times L)$, so we indeed deduce that all continuous functions are $\Sigma\otimes\Upsilon$ measurable. This completes the proof.  
\end{proof}

The indexed family $(\nu_t)_{t\in K}$ provided by the previous lemma will be called a \emph{disintegration kernel of $\nu$}. In case $L$ is metrizable, the disintegration kernel is essentially unique and has some additional properties, collected in the following lemma.

\begin{lemma}\label{L:dezintegrace metriz}
  Let $K,L,\nu,\sigma$ be as in Lemma~\ref{L:dezintegrace}. Assume moreover that $L$ is metrizable.
  Then the following assertions are valid.
  \begin{enumerate}[$(a)$]
      \item Assume that $(\nu_t)_{t\in K}$ is an indexed family of complex Radon measures on $L$ satisfying conditions $(i)$ and $(ii)$ from Lemma~\ref{L:dezintegrace}. Then it is a disintegration kernel of $\nu$.
      \item If $(\nu_t)_{t\in K}$ is a disintegration kernel of $\nu$, then the mapping $t\mapsto \nu_t$ is $\sigma$-measurable as a mapping from $K$ to $(B_{M(K,\ce)},w^*)$.
      \item The disintegration kernel of $\nu$ is uniquely determined up to a set of $\sigma$-measure zero.
  \end{enumerate}
\end{lemma}

\begin{proof}
    Assume that $(\nu_t)_{t\in K}$ is an indexed family of complex Radon measures on $L$ satisfying conditions $(i)$ and $(ii)$ from Lemma~\ref{L:dezintegrace}. Given $f\in C(L)$, the function $(t,z)\mapsto f(z)$ is continuous on $K\times L$, so condition $(ii)$ shows that the function
    $$t\mapsto \int f\di\nu_t$$
    is $\sigma$-measurable. Since $L$ is metrizable, the space $C(L)$ is separable and hence $(B_{C(L)^*},w^*)$ is metrizable, hence second countable. It easily follows that the mapping $t\mapsto \nu_t$ is $\sigma$-measurable. Hence, assertion $(b)$ follows.

    Further, let $(\nu^\prime_t)_{t\in K}$ be another  indexed family of complex Radon measures on $L$ satisfying conditions $(i)$ and $(ii)$ from Lemma~\ref{L:dezintegrace}. Fix $f\in C(L)$.
    As in the previous paragraph we get that the functions
    $$h(t)=\int f\di\nu_t\mbox{ and }h^\prime(t)=\int f\di\nu^\prime_t,\quad t\in K,$$
    are $\sigma$-measurable. Moreover, for any $g\in C(K)$ the function $(t,z)\mapsto g(t)f(z)$ is continuous on $K\times L$ and hence we get (using condition $(ii)$)
    $$\int gh\di\sigma =\int g(t)f(z)\di\nu(t,z)=\int gh^\prime\di\sigma.$$
    Since this holds for any $g\in C(K)$, we deduce that $h=h^\prime$ $\sigma$-almost everywhere.
    I.e., 
    \[
     \int f\di\nu_t= \int f\di\nu^\prime_t\mbox{ for $\sigma$-almost all $t\in K$ whenever }f\in C(L).
    \]
    Since $C(L)$ is separable, we easily deduce that $\nu_t=\nu^\prime_t$ for $\sigma$-almost all $t\in K$. Assertion $(c)$ now easily follows. 
    
    Assertion $(a)$ follows as well. Indeed, it is enough to apply the above reasoning to a family $(\nu_t)_{t\in K}$ satisfying conditions $(i)$ and $(ii)$ and a disintegration kernel $(\nu^\prime_t)_{t\in K}$ which exists due to  Lemma~\ref{L:dezintegrace}.
\end{proof}

We continue by the following lemma which will be used to combine disintegration with separable reduction methods.

\begin{lemma}\label{L:dezintegrace kvocient}
    Let $K,L,\nu,\sigma$ be as in Lemma~\ref{L:dezintegrace}. Assume that $\nu\ge0$. Let $L^\prime$ be a metrizable compact space and let $\varphi:L\to L^\prime$ be a continuous surjection. Let $\nu^\prime=(\id\times \varphi)(\nu)$ be the image of $\nu$ under the mapping $\id\times\varphi$. If $(\nu_t)_{t\in K}$ is a disintegration kernel of $\nu$, then $(\varphi(\nu_t))_{t\in K}$ is the disintegration kernel of $\nu^\prime$.
\end{lemma}

\begin{proof}
    Since $\nu\ge0$, $\sigma$ is the projection of $\abs{\nu}=\nu$ and simultaneously the projection of $\abs{\nu^\prime}=\nu^\prime$. Moreover, each $\varphi(\nu_t)$ is a probability. To prove that $(\varphi(\nu_t))_{t\in K}$ is a disintegration kernel of $\nu^\prime$ it is enough, due to Lemma~\ref{L:dezintegrace metriz}$(a)$, to verify condition $(ii)$ from Lemma~\ref{L:dezintegrace}. So, fix $f\in C(K\times L^\prime)$. We have
    $$\begin{aligned}
        \int_{K\times L^\prime} f\di\nu^\prime&=\int_{K\times L} f\circ (\operatorname{id}\times \varphi)\di\nu=\int_K\left(\int_L f(t,\varphi(z))\di\nu_t(z)\right)\di\sigma(t)
        \\&=\int_K\left(\int_{L^\prime} f(t,y)\di\varphi(\nu_t)(y)\right)\di\sigma(t),
    \end{aligned}$$
    where we used Lemma~\ref{L:dezintegrace} and the rules of integration with respect to the image of a measure. This completes the proof.
\end{proof}

In case $L$ is not metrizable, the question of uniqueness is more delicate. In particular, it is not hard to construct counterexamples showing that assertion $(a)$ of Lemma~\ref{L:dezintegrace metriz} may fail for nonmetrizable $L$ (for example if $L=[0,1]^{[0,1]}$ or if $L$ is the ordinal interval $[0,\omega_1]$). However, there is a substitute for uniqueness in the general case which is contained in the following proposition.

\begin{prop}\label{P:vsude}
    Let $K$ and $L$ be two compact Hausdorff spaces and let $\sigma$ be a positive Radon measure on $K$. Let
    $$M=\{\nu\in M_+(K\times L)\setsep \pi_1(\nu)=\sigma\}.$$
    Then there is an assignment of disintegration kernels
    $$\nu\in M\mapsto (\nu_t)_{t\in K}$$
   such that for any two measures $\nu_1,\nu_2\in M$ and any two bounded Borel functions $g_1,g_2:L\to\er$ we have
   \begin{multline*}       
   \int g_1\di\nu_{1,t}\le \int g_2\di\nu_{2,t} \quad \sigma\mbox{-almost everywhere} \\
   \implies \int g_1\di\nu_{1,t}\le \int g_2\di\nu_{2,t} \quad \mbox{for each }t\in K.\end{multline*}
\end{prop}

\begin{proof}
    Let $\Sigma$ denote the $\sigma$-algebra of $\sigma$-measurable subsets of $K$. Let $\Phi_0:\Sigma\to\Sigma$ be a lifting (in the sense of \cite[Definition 341A]{fremlin3}) provided by \cite[Theorem 341K]{fremlin3}. By \cite[Theorem 363F and Exercise 363Xe]{fremlin3} this mapping induces a linear lifting $\Phi:L^\infty(\sigma)\to L^\infty(\Sigma)$ (where $L^\infty(\Sigma)$ is the space of all bounded $\Sigma$-measurable functions on $K$ equipped with the supremum norm) which is also an order isomorphism and satisfies $\norm{\Phi}\le1$.

    Given $\nu\in M$ and $A\subset L$ Borel, the assignment
    $$\nu_A(B)=\nu(B\times A),\quad B\in\Sigma,$$
    is a measure on $(K,\Sigma)$ satisfying $\nu_A\le\sigma$. Let $h_A$ denote the Radon-Nikod\'ym derivative of $\nu_A$ with respect to $\sigma$. It follows from the proof of \cite[Theorem 452M]{fremlin4} that the formula
      $$\nu_t(A)=\Phi(h_A)(t),\quad A\subset L\mbox{ Borel}, t\in K,$$
     provides a disintegration kernel of $\nu$. We will show that this is a correct choice. 

     Given a bounded Borel function $g:L\to\er$, the formula
     $$\nu_g(B)=\int_{B\times L} g(z)\di\nu(t,z),\quad B\in\Sigma,$$
     defines a signed measure on $(K,\Sigma)$. Moreover, for each $B\in\Sigma$ we have 
     $$\abs{\nu_g(B)}\le \norm{g}_\infty\cdot \nu(B\times L)=\norm{g}_\infty\cdot\sigma(B).$$
     By the definition of absolute variation we easily get $\abs{\nu_g}\le \norm{g}_\infty\cdot\sigma$. In particular, $\nu_g$ is absolutely continuous with respect to $\sigma$ and its Radon-Nikod\'ym density $h_g$ satisfies $\norm{h_g}_\infty\le\norm{g}_\infty$. It follows that the assignment $g\mapsto h_g$ is a nonexpansive linear operator from the space of bounded Borel functions on $L$ into $L^\infty(\sigma)$.

     We claim that for each bounded Borel function $g$ on $L$ we have
     $$\int g\di\nu_t=\Phi(h_g)(t), \quad t\in K.$$
     Fix $t\in K$. By the choice of $\nu_t$, the equality holds if $g$ is a characteristic function of a Borel set. By linearity and continuity we deduce that it holds for each bounded Borel function.

     Now assume that $\nu_1,\nu_2\in M$ and $g_1,g_2:L\to \er$ are two bounded Borel functions satisfing
     $$\int g_1\di\nu_{1,t} \le \int g_2\di\nu_{2,t}\qquad \sigma\mbox{-almost everywhere}.$$
     Thus
     $$\nu_{1,g_1}(B)=\int_B\left(\int_L g_1\di\nu_{1,t}\right)\di\sigma(t)\le \int_B\left(\int_L g_2\di\nu_{2,t}\right)\di\sigma(t)=\nu_{2,g_2}(B)$$
     for each $B\in\A$, i.e., $\nu_{1,g_1}\le \nu_{2,g_2}$. It follows that $h_{1,g_1}\le h_{2,g_2}$ in $L^\infty(\sigma)$ and hence also $\Phi(h_{1,g_1})\le\Phi(h_{2,g_2})$ in $L^\infty(\Sigma)$. 
     This completes the proof.
\end{proof}

\section{The Hustad mapping via disintegration}\label{s:tovector}

In this section we analyze in more detail the operator $T^*$ interpreted as a mapping from $M(K\times B_{E^*})$ to $M(K,E^*)$. Recall that this operator is defined by  formula \eqref{eq:hustad}. We start by a slight strengthening of \cite[Lemma 3.1]{batty-vector}.

\begin{lemma}\label{L:nesena sferou}
    If $\nu\in M(K\times B_{E^*})$ satisfies $\norm{T^*\nu}=\norm{\nu}$, then $\nu$ is carried by $K\times S_{E^*}$ (here $S_{E^*}$ denotes the dual unit sphere).
\end{lemma}

\begin{proof}
    Assume that $\norm{T^*\nu}=\norm{\nu}$. Then
    $$\begin{aligned}
      \norm{\nu}&=\norm{T^*\nu}=\sup \left\{\abs{\int\f\di T^*\nu}\setsep \f\in C(K,E), \norm{\f}\le 1\right\}
      \\&=\sup \left\{\abs{\int x^*(\f(t))\di\nu(x^*,t)}\setsep \f\in C(K,E), \norm{\f}\le 1\right\}
     \\ &\le \sup \left\{\int \abs{x^*(\f(t))}\di\abs{\nu}(x^*,t)\setsep \f\in C(K,E), \norm{\f}\le 1\right\}
     \\&\le \int \norm{x^*} \di\abs{\nu}(x^*,t) \le \int 1 \di\abs{\nu}(x^*,t) =\norm{\nu},
    \end{aligned}$$
    so the equalities hold. In particular, $\norm{x^*}=1$ $\abs{\nu}$-a.e., which completes the proof.
\end{proof}

We will further strengthen this lemma using disintegration. To this end we will need the following simple fact.

\begin{lemma}\label{L:teziste}
    Let $\nu$ be an $\ef$-valued Radon measure on $B_{E^*}$. Then there is a unique point $r(\nu)\in E^*$ such that
    $$\int x^*(x)\di\nu(x^*)= r(\nu)(x)\mbox{ for each }x\in E.$$
    Moreover, $\norm{r(\nu)}\le\norm{\nu}$. If $\nu$ is a probability measure, then $r(\nu)$ is the barycenter of $\nu$.
\end{lemma}

\begin{proof}
    It is obvious that the mapping $x\mapsto \int x^*(x)\di\nu(x^*)$ is a linear functional on $E$ of norm at most $\norm{\nu}$. Moreover, if $\nu$ is a probability, then the equality is clearly satisfied for the barycenter.
\end{proof}

 We continue by providing a formula for $T^*\nu$ using a kind of `density function'.

\begin{prop}\label{P:hustotaT*nu}
    Let $\nu\in M(K\times B_{E^*})$ be arbitrary. Then there is a function $\h:K\to B_{E^*}$ such that 
    $$\int_K \f\di T^*\nu=\int_K \h(t)(\f(t))\di\pi_1(\abs{\nu})(t)\mbox{ for }\f\in C(K,E).$$
    Then we also have
      $$T^*\nu(A)(x)=\int_A \h(t)(x)\di\pi_1(\abs{\nu})(t)\mbox{ for }A\subset K\mbox{ Borel and }x\in E.$$ 
     A possible choice for $\h$ is $\h(t)=r(\nu_t)$ for $t\in K$, where $(\nu_t)_{t\in K}$ is a disintegration kernel for $\nu$. 
\end{prop}

\begin{proof} To simplify the notation we set $\sigma=\pi_1(\abs{\nu})$.  Let $(\nu_t)_{t\in K}$ be a disintegration kernel of $\nu$. For each $t\in K$ let $\h(t)=r(\nu_t)\in B_{E^*}$ be the functional provided by Lemma~\ref{L:teziste}. Let us now prove that $\h$ satisfies the first assertion:

 By the definition of $T^*$ and Lemma~\ref{L:dezintegrace} we have
$$\begin{aligned}
    \int_K \f\di T^*\nu& = \int_{K\times B_{E^*}} x^*(\f(t))\di\nu(t,x^*)
    =\int_K\left(\int_{B_{E^*}} x^*(\f(t))\di\nu_t(x^*)\right)\di\sigma(t)
    \\&=\int_K r(\nu_t)(\f(t))\di\sigma(t)=\int_K \h(t)(\f(t))\di\sigma(t),
\end{aligned}$$
where in the third equality we used the choice of $r(\nu_t)$.

We proceed by deducing the second assertion from the first one. Fix $x\in E$. For each $f\in C(K)$ we have
$$\int_K f\di (T^*\nu)_x=\int_K f\cdot x \di T^*\nu=\int_K f(t)\cdot \h(t)(x)\di\sigma(t).$$
By the Lebesgue dominated convergence we may extend this equality to bounded Baire functions on $K$.
Therefore the second assertion holds for any Baire set $A\subset K$. By regularity of the measures in question this may be extended to Borel sets.
\end{proof}

The function $\h$ from the previous proposition is a kind of weak$^*$ Radon-Nikod\'ym density of $T^*\nu$ with respect to $\pi_1(\abs{\nu})$. Note that it need not be measurable, but it is weak$^*$ measurable, i.e., $t\mapsto \h(t)(x)$ is $\pi_1(\abs{\nu})$-measurable for each $x\in E$. We will see in Proposition~\ref{P:vlastnostih} below that stronger properties are satisfied if $E$ is separable. 

In general, the function $\h$ is not uniquely determined -- it is not hard to 
find a nonseparable $E$ and $\nu$ such that there are two everywhere different functions $\h_1$ and $\h_2$ with the required properties. However, we have the following partial uniqueness result.

\begin{lemma}\label{L:hustota jednoznacnost}
    Let $\nu\in M(K\times B_{E^*})$ be arbitrary. Let $\h_1$ and $\h_2$ be two functions satisfying conditions from Proposition~\ref{P:hustotaT*nu}.
    \begin{enumerate}[$(a)$]
        \item Let $F\subset E$ be a separable subspace. Then $\h_1(t)|_F=\h_2(t)|_F$ $\pi_1(\abs{\nu})$-almost everywhere.
        \item If $E$ is separable, then  $\h_1(t)=\h_2(t)$ $\pi_1(\abs{\nu})$-almost everywhere. 
        In particular, if $(\nu_t)_{t\in K}$ is a disintegration kernel of $\nu$, then $\h_1(t)=r(\nu_t)$  $\pi_1(\abs{\nu})$-almost everywhere.
    \end{enumerate}
\end{lemma}

\begin{proof}
    Fix $x\in E$. By Proposition~\ref{P:hustotaT*nu} we have
    $$\int_A \h_1(t)(x)\di\pi_1(\abs{\nu})(t)=\int_A \h_2(t)(x)\di\pi_1(\abs{\nu})(t)\mbox{ for }A\subset K\mbox{ Borel}.$$
    Hence $\h_1(t)(x)=\h_2(t)(x)$ for $\pi_1(\abs{\nu})$-almost all $t\in K$. Now both assertions easily follow. 
\end{proof}

We continue by more detailed analysis of the `density function'. The following proposition provides, among others, the promised strengthening of Lemma~\ref{L:nesena sferou}.

\begin{prop}\label{P:vlastnostih}
     Let $\nu\in M(K\times B_{E^*})$ be arbitrary and let $\h:K\to B_{E^*}$ be the function provided by Proposition~\ref{P:hustotaT*nu}. Then the following assertions are valid:
     \begin{enumerate}[$(a)$]
         \item $\abs{T^*\nu}\le \pi_1(\abs{\nu})$. If $\norm{T^*\nu}=\norm{\nu}$, then $\abs{T^*\nu}= \pi_1(\abs{\nu})$.
         \item If $\norm{T^*\nu}=\norm{\nu}$, then $\norm{\h(t)}=1$  $\pi_1(\abs{\nu})$-almost everywhere. If, in addition, $(\nu_t)_{t\in K}$ is a disintegration kernel of $\nu$, then $r(\nu_t)\in S_{E^*}$           $\pi_1(\abs{\nu})$-almost everywhere.
         \item If $E$ is separable, then $\h$ is a $\pi_1(\abs{\nu})$-measurable function from $K$ to $(B_{E^*},w^*)$ and $$\norm{T^*\nu}=\int_K\norm{\h(t)}\di \pi_1(\abs{\nu})(t).$$
         \item If $F$ is a separable subspace of $E$, then $t\mapsto \h(t)|_F$ is a $\pi_1(\abs{\nu})$-measurable function from $K$ to $(B_{F^*},w^*)$. Moreover,
         $$\norm{T^*\nu}=\max \left\{\int_K\norm{\h(t)|_F}\di \pi_1(\abs{\nu})(t)\setsep F\subset E\mbox{ separable}\right\}.$$
     \end{enumerate}
\end{prop}

\begin{proof} To simplify the notation we again set $\sigma=\pi_1(\abs{\nu})$. 

$(a)$: Given $A\subset K$ Borel and $x\in E$, Proposition~\ref{P:hustotaT*nu} yields
$$\begin{aligned}
  \abs{T^*\nu(A)(x)}&=\abs{\int_A \h(t)(x)\di\sigma(t)}\le
  \int_A \abs{\h(t)(x)}\di\sigma(t)\le \norm{x}\sigma(A),
\end{aligned}$$
hence $\norm{T^*\nu(A)}\le\sigma(A)$. Now it easily follows that $\abs{T^*\nu}\le\sigma$.
If $\norm{T^*\nu}=\norm{\nu}$, then $\norm{T^*\nu}=\norm{\sigma}$ (as clearly $\norm{\nu}=\norm{\sigma}$) and hence $\abs{T^*\nu}=\sigma$.

$(b)$:  We proceed by contraposition. Assume that the set $\{t\in K\setsep\norm{\h(t)}<1\}$ is not of $\sigma$-measure zero. It follows that there is some $c<1$ such that the set $$A=\{t\in K\setsep \norm{\h(t)}\le c\}$$
has positive outer measure. Set $\delta=\sigma^*(A)$ (note that $\sigma^*$ denotes the outer measure induced by $\sigma$). Fix any $\f\in B_{C(K,E)}$. Then the set
$$C=\{t\in K\setsep \abs{\h(t)(\f(t))}\le c\}$$ 
is $\sigma$-measurable and contains $A$ (if $t\in A$, then $ \abs{\h(t)(\f(t))}\le\norm{\h(t)}\norm{\f(t)}\le c$). Hence
$$\begin{aligned}
   \abs{\int_K \f\di T^*\nu} & =\abs{\int_K \h(t)(\f(t))\di\sigma(t)}\le \int_K \abs{\h(t)(\f(t))}\di\sigma(t) \\&\le c\sigma(C)+\sigma(K\setminus C)
   = \norm{\sigma} + (c-1)\sigma(C)\\&\le \norm{\nu}+(c-1)\delta.
\end{aligned}$$
Hence
$$\norm{T^*\nu}\le \norm{\nu}+(c-1)\delta<\norm{\nu},$$
which completes the argument. The additional statement follows from Proposition~\ref{P:hustotaT*nu}.

$(c)$: Assume $E$ is separable. Then $(B_{E^*},w^*)$ is a compact metrizable space, hence it is second countable. By the assumption we know that $t\mapsto\h(t)(x)$ is $\sigma$-measurable for each $x\in E$. It follows that $\h^{-1}(U)$ is $\sigma$-measurable whenever $U$ belongs to the canonical
base of the weak$^*$ topology of $B_{E^*}$. By second countability this may be extended to any weak$^*$ open set, so $\h$ is $\sigma$-measurable.

Hence also $t\mapsto\norm{\h(t)}$ is $\sigma$-measurable. Moreover, if $\f\in B_{C(K,E)}$, then
$$\begin{aligned}
\abs{\int_K \f\di T^*\nu}&=\abs{\int_K \h(t)(\f(t))\di\sigma(t)}
\le \int_K \abs{\h(t)(\f(t))}\di\sigma(t) \\&\le \int_K \norm{\h(t)}\norm{\f(t)}\di\sigma(t)
\le  \int_K \norm{\h(t)}\di\sigma(t),
\end{aligned}$$
so $$\norm{T^*\nu}\le  \int_K \norm{\h(t)}\di\sigma(t).$$

To prove the converse inequality fix $\ep>0$. For $x^*\in S_{E^*}$ set
$$\psi(x^*)=\{x\in B_E\setsep \Re x^*(x)>1-\ep\}.$$
Then $\psi(x^*)$ is a nonempty convex set. Moreover, the set-valued mapping $\psi$ is clearly lower semicontinuous from the weak$^*$ topology to the norm. Since $(S_{E^*},w^*)$ is a separable completely metrizable space and the mapping  $x^*\mapsto\overline{\psi(x^*)}$ is also lower semicontinuous (cf. \cite[Proposition 2.3]{michael}), Michael's selection theorem \cite[Theorem 3.2$^{\prime\prime}$]{michael} provides a continuous selection of this mapping. Hence, we have a (weak$^*$-to-norm) continuous map $\g:S_{E^*}\to B_E$ such that $\Re x^*(\g(x^*))\ge 1-\ep$ for each $x^*\in S_{E^*}$. Define a mapping $\f_0:K\to B_E$ by
$$\f_0(t)=\begin{cases}
    \g(\frac{\h(t)}{\norm{\h(t)}}),& \h(t)\ne0,\\ 0, & \h(t)=0.
\end{cases}$$
Then $\f_0$ is $\sigma$-measurable and $\Re\h(t)(\f_0(t))\ge(1-\ep)\norm{\h(t)}$ for $t\in K$.

By Luzin's theorem (see \cite[Theorem 418J and Definition 411M]{fremlin4}) there is a closed subset $B\subset K$ such that $\sigma(K\setminus B)<\ep$ and $\f_0|_B$ is continuous. By a consequence of Michael's selection theorem (cf. \cite[Corollary 1.5]{michael} or \cite[Theorem 3.1]{michael}) there is a continuous function $\f:K\to B_{E}$  extending $\f_0$. Then
$$\begin{aligned}
    \norm{T^*\nu}&\ge\abs{\int_K \f\di T^*\nu}=\abs{\int_K \h(t)(\f(t))\di\sigma(t)}
    \ge \abs{\int_B \h(t)(\f(t))\di\sigma(t)} -\ep \\
    &=\abs{\int_B \h(t)(\f_0(t))\di\sigma(t)} -\ep \ge \int_B \Re \h(t)(\f_0(t))\di\sigma(t)-\ep
    \\& \ge \int_B (1-\ep)\norm{\h(t)}\di\sigma(t) -\ep \ge (1-\ep)\left(\int_K\norm{\h(t)}\di\sigma(t)-\ep\right)-\ep.
\end{aligned}$$
Since $\ep>0$ is arbitrary, the remaining inequality follows.

$(d)$: Fix a separable subspace $F\subset E$. For each $x\in F$ we have
$$\h(t)|_F(x)=\h(t)(x), \quad t\in K,$$
so the function $t\mapsto \h(t)|_F$ is $\sigma$-measurable by the argument used in the proof of $(c)$.
We also get
$$\norm{T^*\nu}\ge\norm{T^*\nu|_{C(K,F)}}=\int_K \norm{\h(t)|_F}\di\sigma(t),$$
again by the same arguments as in $(c)$.

Conversely, there is a sequence $(\f_n)$ in $B_{C(K,E)}$ such that
$$\norm{T^*\nu}=\sup\left\{\abs{\int_K \f_n\di T^*\nu}\setsep n\in\en\right\}.$$
Since $\f_n(K)$ is a compact (hence separable) subset of $E$ for each $n\in\en$, there is a separable
subspace $F\subset E$ such that $\f_n(K)\subset F$ for $n\in\en$. Then
$$\norm{T^*\nu}=\norm{T^*\nu|_{C(K,F)}}=\int_K \norm{\h(t)|_F}\di\sigma(t),$$
since $\h(t)|_F$ is a density of $(T^*\nu)|_{C(K,F)}$.
This completes the proof.
\end{proof}

The next lemma provides a more precise version of assertion $(a)$ from the previous proposition
by describing the Radon-Nikod\'ym density $\abs{T^*\nu}$ with respect to $\pi_1(\abs{\nu})$.

\begin{lemma}\label{L:hustotavariace}
    Let $\nu\in M(K\times B_{E^*})$ be arbitrary, let $\h:K\to B_{E^*}$ be the function provided by Proposition~\ref{P:hustotaT*nu} and let $F_0\subset E$ be a separable subspace at which the maximum from Proposition~\ref{P:vlastnostih}$(d)$ is attained. Then the following assertions are valid:
    \begin{enumerate}[$(a)$]
        \item If $F\subset E$ is a separable subspace containing $F_0$, then $\norm{\h(t)|_F}=\norm{\h(t)|_{F_0}}$ $\pi_1(\abs{\nu})$-a.e.
        \item $\di \abs{T^*\nu}= \norm{\h|_{F_0}}\di\pi_1(\abs{\nu})$, i.e., 
        $$\abs{T^*\nu}(A)=\int_A \norm{\h(t)|_{F_0}}\di\pi_1(\abs{\nu})\mbox{ for }A\subset K\mbox{ Borel}.$$
    \end{enumerate}
\end{lemma}

\begin{proof} We again set $\sigma=\pi_1(\abs{\nu})$.

$(a)$: Since $F\supset F_0$, $\norm{\h(t)|_F}\ge\norm{\h(t)_{F_0}}$ everywhere. On the other hand, by Proposition~\ref{P:vlastnostih}$(d)$ these two functions are $\sigma$-measurable and have the same integral with respect to $\sigma$. Thus they are equal $\sigma$-almost everywhere.

$(b)$: Fix $A\subset K$ Borel and $x\in E$. Let $F=\span(F_0\cup\{x\})$. By Proposition~\ref{P:hustotaT*nu} we get
$$\begin{aligned}
  \abs{T^*\nu(A)(x)}&=\abs{\int_A \h(t)(x) \di\sigma(t)}\le\int_A \abs{\h(t)(x)} \di\sigma(t)
  \le \int_A \norm{\h(t)|_F}\norm{x} \di\sigma(t) \\&=\norm{x} \int_A \norm{\h(t)|_{F_0}}\di\sigma(t),
\end{aligned}$$
where the last equality follows from $(a)$. Therefore
$$\norm{T^*\nu(A)}\le \int_A \norm{\h(t)|_{F_0}}\di\sigma(t).$$
By definition of the absolute variation we deduce
$$\abs{T^*\nu}(A)\le \int_A \norm{\h(t)|_{F_0}}\di\sigma(t).$$
Finally, using the choice of $F_0$ we deduce that the equality holds.
\end{proof}

The final result of this section provides a construction replacing any $\nu\in M(K\times B_{E^*})$ by a positive measure in a canonical way. This will serve as starting point for the next section
devoted to an alternative view to Batty's procedure of transference of measures. 
Before coming to the final result, we give a simple consequence of the Stone-Weierstrass theorem.

\begin{lemma}\label{L:SW}
    The closed self-adjoint subalgebra of $C(K\times B_{E^*})$ generated by $T(C(K,E))$ is
    $$C^0(K\times B_{E^*})=\{f\in C(K\times B_{E^*})\setsep f|_{K\times\{0\}}=0\}.$$
\end{lemma}

\begin{proof}
    Inclusion `$\subset$' is obvious. To prove the converse we use the Stone-Weierstrass theorem.
    Assume that $(t,x^*),(s,y^*)\in K\times B_{E^*}$. Then:
    \begin{itemize}
        \item Assume $y^*\ne x^*$. Fix $x\in E$ with $y^*(x)\ne x^*(x)$. Let $\f\in C(K,E)$ be the constant function equal to $x$. Then $T\f(t,x^*)=x^*(x)\ne y^*(x)=T\f(s,y^*)$.
        \item Assume $y^*=x^*\ne0$ and $s\ne t$. Fix $x\in E$ with $x^*(x)\ne0$ and $f\in C(K)$ with $f(s)=0$ and $f(t)=1$. Then
        $$T(f\cdot x)(t,x^*)=x^*(x)\ne0=T(f\cdot x)(s,y^*).$$
    \end{itemize}
    Let $Z$ denote the algebra from the statement. By the Stone-Weierstrass theorem we deduce
     $$\span(Z\cup\{1\})=\{f\in C(K\times B_{E^*})\setsep f|_{K\times\{0\}}\mbox{ is constant}\},$$
    hence, $Z=C^0(K\times B_{E^*})$.
\end{proof}

\begin{prop}\label{P:KT*nu konstrukce}
    Let $\nu,\h,F_0$ be as in Lemma~\ref{L:hustotavariace}. Define
    $$\g(t)=\begin{cases}
        \frac{\h(t)}{\norm{\h(t)|_{F_0}}},& \h(t)|_{F_0}\ne0,\\ 0 &\mbox{otherwise}.
    \end{cases}$$
    Then the following assertions hold:
    \begin{enumerate}[$(i)$]
        \item There is a unique measure $\widetilde{\nu}\in M(K\times B_{E^*})$ carried by
        $K\times (B_{E^*}\setminus\{0\})$ such that
         $$\int_{K\times B_{E^*}} f\di\widetilde{\nu} = \int_K f(t,\g(t))\di \abs{T^*\nu}(t)\mbox{ for }f\in C^0(K\times B_{E^*}).$$
         \item The measure $\widetilde{\nu}$ is positive, $T^*\widetilde{\nu}=T^*\nu$ and $\norm{\widetilde{\nu}}=\norm{T^*\nu}$. In particular, $\widetilde{\nu}$ is carried by $K\times S_{E^*}$.
         \item $\pi_1(\widetilde{\nu})=\abs{T^*\widetilde{\nu}}=\abs{T^*\nu}$.
         \item $\norm{\g(t)}=1$ (hence $\norm{\h(t)}=\norm{\h(t)|_{F_0}}$) $\pi_1(\widetilde{\nu})$-a.e. 
    \end{enumerate}
\end{prop}

\begin{proof}
  Let $\f\in C(K,E)$. Let $F\subset E$ be a separable subspace containing $F_0$ and $\f(K)$. Then we have
 $$\begin{aligned}
     \int_K \f\di T^*\nu&=\int_K \h(t)(\f(t))\di \pi_1(\abs{\nu})(t)=\int_K \g(t)(\f(t)) \norm{\h(t)|_{F_0}}\di\pi_1(\abs{\nu})(t)\\&= \int_K T\f(t,\g(t))\di\abs{T^*\nu}(t).
 \end{aligned}$$
 Here the first equality follows from Proposition~\ref{P:hustotaT*nu}. To prove the second equality we compare the integrated functions. By the definition of $\g$ they are equal if $\h(t)|_{F_0}\ne0$ or $\h(t)(\f(t)) =0$. In the remaining points we have $\h(t)|_{F_0}=0$ and $\h(t)|_F\ne 0$ (recall that $\f(t)\in F$). But such points form a set of $\pi_1(\abs{\nu})$-measure zero by Lemma~\ref{L:hustotavariace}$(a)$. Hence the second equality follows. The third equality follows from the definition of $T$ and Lemma~\ref{L:hustotavariace}$(b)$.

 In particular, we deduce that the function $t\mapsto f(t,\g(t))$ is $\abs{T^*\nu}$-measurable for each $f\in T(C(K,E))$. Since measurability is preserved by products, linear combinations, complex conjugation and limits of sequences, by Lemma~\ref{L:SW} we get that the function $t\mapsto f(t,\g(t))$ is $\abs{T^*\nu}$-measurable for each $f\in C^0(K\times B_{E^*})$. Therefore, the mapping
 $$f\mapsto \int_K f(t,\g(t))\di\abs{T^*\nu}(t)$$
 defines a linear functional on $C^0(K\times B_{E^*})$. It is clear that this functional is bounded, with norm at most $\norm{T^*\nu}$. So, the existence and uniqueness of $\widetilde{\nu}$ follows from the Riesz representation theorem applied to the space $C_0(K\times (B_{E^*}\setminus\{0\}))$. This completes the proof of $(i)$.
 
 By the quoted Riesz theorem the norm of $\widetilde{\nu}$ equals the norm of the represented functional. In particular, $\norm{\widetilde{\nu}}\le\norm{T^*\nu}$. Conversely,
$$\begin{aligned}
 \norm{\widetilde{\nu}}&=\sup\left\{\abs{\int_K f(t,\g(t))\di\abs{T^*\nu}(t)}\setsep f\in C^0(K\times B_{E^*}), \norm{f}_\infty\le 1\right\} \\&
 \ge \sup\left\{\abs{\int_K T\f(t,\g(t))\di\abs{T^*\nu}(t)}\setsep \f\in C(K,E), \norm{\f}_\infty\le 1\right\} \\&
 =\sup\left\{\abs{\int_K \f\di T^*\nu}\setsep \f\in C(K,E), \norm{\f}_\infty\le 1\right\} =\norm{T^*\nu},
\end{aligned}$$
where the first equality on the third line follows from the computation at the beginning of the proof.
Hence $\norm{\widetilde{\nu}}=\norm{T^*\nu}$. Further, $\widetilde{\nu}\ge0$ as the represented functional is clearly positive. Finally, if $\f\in C(K,E)$, then
$$\int \f\di T^*\widetilde{\nu}=\int T\f\di\widetilde{\nu}=\int_K T\f(t,\g(t))\di\abs{T^*\nu}(t)=\int_K \f\di T^*\nu,$$
where the last equality again  follows from the computation at the beginning of the proof. We conclude that $T^*\widetilde{\nu}=T^*\nu$. This completes the proof of $(ii)$ as the last statement follows from Lemma~\ref{L:nesena sferou}. 

Assertion $(iii)$ follows from Proposition~\ref{P:vlastnostih}$(a)$ using $(ii)$. 

$(iv)$: Note that, by the construction, the function $\g$ satisfies the first equality from Proposition~\ref{P:hustotaT*nu} for $\widetilde{\nu}$. Thus we conclude using $(ii)$ and Proposition~\ref{P:vlastnostih}$(b)$.
\end{proof}

\section{Transference of measures revisited}

In this section we show that Proposition~\ref{P:KT*nu konstrukce} provides an alternative approach 
to the procedure named `transference of measures' in \cite[Section 3]{batty-vector}. This procedure provides, given $\mu\in M(K,E^*)$, a canonical positive measure $W\mu\in M(K\times B_{E^*})$ (denoted by $K\mu$ in \cite{batty-vector}) such that $T^*W\mu=\mu$ and $\norm{W\mu}=\norm{\mu}$. The construction in \cite[Section 3]{batty-vector}
uses the correspondence between cones $\B$ and $\D$ recalled in Section~\ref{s:dualita} above. Using our approach we get stronger results than \cite{batty-vector} (as promised in the introduction). We start by the following lemma which may be seen as an ultimate
    generalization of \cite[Proposition 2.2]{batty-vector}.

\begin{lemma}\label{L:KT*nu dukaz}
    Let $\nu\in M(K\times B_{E^*})$ be arbitrary. Let $\widetilde{\nu}$ be the measure provided by Proposition~\ref{P:KT*nu konstrukce}. Then
    $$p_f(T^*\nu)=\int f\di\widetilde{\nu}\mbox{ for each }f\in\D.$$
\end{lemma}

\begin{proof} We adopt the notation from Proposition~\ref{P:KT*nu konstrukce}. The proof will be done in several steps.

\smallskip

{\tt Step 1:} If $f\in\D$, then $f=\inf\{\Re T\f\setsep \f\in C(K,E), \Re T\f\ge f\}$.

\smallskip

By Corollary~\ref{cor:f-p} we have (for $(t,x^*)\in K\times E^*$)
$$\begin{aligned}  
f(t,x^*)&=p_f(\ep_t\otimes x^*)\\&=\inf \Big\{ \Re \int \f\di(\ep_t\otimes x^*)\setsep \f\in C(K,E), \\& \qquad\qquad \Re y^*(\f(s))\ge f(s,y^*)\mbox{ for }(s,y^*)\in K\times E^*\Big\}
\\&=\inf\{ \Re T\f(t,x^*)\setsep \f\in C(K,E), \Re T\f\ge f\},
\end{aligned}$$
which completes the argument.

\smallskip

{\tt Step 2:} $p_f(T^*\nu)\ge \int f\di\widetilde{\nu}$ for each $f\in\D$:

\smallskip

Let $\f\in C(K,E)$ be arbirary. Then
$$\begin{aligned}
 \int T\f\di\widetilde{\nu}&=\int T\f(t,\g(t)) \di \abs{T^*\nu}(t)=\int \g(t)(\f(t)) \di \abs{T^*\nu}(t)\\&=\int \g(t)(\f(t))  \norm{\h(t)|_{F_0}} \di \pi_1(\abs{\nu})(t)=
\int \h(t)(\f(t)) \di  \pi_1(\abs{\nu})(t) \\& =\int \f\di T^*\nu.\end{aligned}$$
Indeed, the first equality follows from Proposition~\ref{P:KT*nu konstrukce} as $T\f\in C^0(K\times B_{E^*})$. The second equality follows from the definition of $T$, the third one from Lemma~\ref{L:hustotavariace}$(b)$ ($F_0$ has the same meaning as in Lemma~\ref{L:hustotavariace}). Let us explain the fourth equality. Let $F\subset E$ be a separable subspace containing $F_0\cup \f(K)$. Then $\h(t)(\f(t))=\g(t)(\f(t))\norm{\h(t)|_{F_0}}$ unless $\h(t)(\f(t))\ne 0$ and $\h(t)|_{F_0}=0$. But such points form a set of $\pi_1(\abs{\nu})$-measure zero by Lemma~\ref{L:hustotavariace}$(a)$. The last equality follows from Proposition~\ref{P:hustotaT*nu}.

Let $f\in \D$ be arbitrary. Using Step 1 and the preceding computation we get
$$\begin{aligned}
\int f\di\widetilde{\nu}&=\int \inf \{\Re T\f \setsep \f\in C(K,E), \Re T\f\ge f\}\di\widetilde{\nu}\\&\le  \inf \left\{\Re \int T\f \di\widetilde{\nu} \setsep \f\in C(K,E), \Re T\f\ge f\right\}
\\&=\inf \left\{\Re \int \f \di T^*\nu \setsep \f\in C(K,E), \Re T\f\ge f\right\}
=p_f(T^*\nu),\end{aligned}$$
where the last equality follows from Corollary~\ref{cor:f-p}.

\smallskip

{\tt Step 3:} If $f=\Re T\f$ for some $\f\in C(K,E)$, then $p_f(T^*\nu)=\int f\di\widetilde{\nu}$.

\smallskip

Assume $f=\Re T\f$. It follows from Proposition~\ref{P:BD-BD} and from the first computation in Step 2 that
$$p_{f}(T^*\nu)=\Re \int \f\di T^*\nu =\int \Re T\f\di\widetilde{\nu}=\int f\di\widetilde{\nu}.$$

{\smallskip}

{\tt Step 4.} Assume $\f_1,\dots,\f_n\in C(K,E)$, $f_j=\Re T\f_j$ and $f=f_1\wedge\dots\wedge f_n$.
Then $p_f(T^*\nu)=\int f\di\widetilde{\nu}$.

\smallskip


Assume first that $f(t,\g(t))=f_1(t,\g(t))$ $\abs{T^*\nu}$-almost everywhere. Then
$$\begin{aligned}
    p_f(T^*\nu) &\le p_{f_1}(T^*\nu) =\int f_1\di\widetilde{\nu}=\int f\di\widetilde{\nu} \le p_f(T^*\nu).
\end{aligned}$$
The first inequality follows from the fact that $f\le f_1$. The equalities  follow from Step 3 and from the definition of $\widetilde{\nu}$. The last inequality follows from Step 2.
Hence the equalities hold.

In general there is a partition of $K$ into Borel sets $A_1,\dots,A_n$ such that $f(t,\g(t))=f_j(t,\g(t))$ for $\abs{T^*\nu}$-almost all $t\in A_j$. Then
$$\begin{aligned}
    p_f(T^*\nu)&=\sum_{j=1}^n p_f(T^*\nu|_{A_j})=
    \sum_{j=1}^n p_f(T^*(\nu|_{A_j\times B_{E^*}}))= 
           \sum_{j=1}^n \int_K f\di (\widetilde{\nu|_{A_j\times B_{E^*}}})\\&= \sum_{j=1}^n \int_K f\di \widetilde{\nu}|_{A_j\times B_{E^*}}=\int_K f\di\widetilde{\nu}.
\end{aligned}$$
Here the first equality follows from the fact that $p_f$ is additive on pairs of mutually orthogonal measures. The second one follows easily from formula \eqref{eq:hustad}. The fourth equality follow from 
the construction of $\widetilde{\nu}$ -- it is clear that this measure constructed for $\nu|_{A_j\times B_{E^*}}$ coincides with $\widetilde{\nu}$ restricted to $A_j\times B_{E^*}$. In view of this the third equality follows from the special case addressed in the previous paragraph. The last equality is obvious.

\smallskip

{\tt Step 5.} The general case:

Let now $f\in\D$ be general. Then
 $$\begin{aligned}
\int f\di\widetilde{\nu}&=\int \inf \{h_1\wedge\dots\wedge h_n\setsep h_j\in \D\cap(-\D), h_j\ge f \mbox{ for }1\le j\le n\}\di\widetilde{\nu}\\&=\inf \left\{\int (h_1\wedge\dots\wedge h_n)\di\widetilde{\nu}\setsep h_j\in \D\cap(-\D), h_j\ge f\mbox{ for }1\le j\le n\right\}\\&=
\inf\{p_{h_1\wedge\dots\wedge h_n}(T^*\nu)\setsep h_j\in \D\cap(-\D), h_j\ge f\mbox{ for }1\le j\le n\}=p_f(T^*\nu).
\end{aligned}$$
Here the first equality follows from Step 1 and the description of $\D\cap(-\D)$ in Proposition~\ref{P:BD-BD}. The second one follows from the monotone convergence theorem for nets. The third equality follows from Step 4. The last equality follows easily from Corollary~\ref{cor:f-p}.
\end{proof}

Now we easily get the promised relationship to the Batty's operator:

\begin{cor}\label{c:battytrans}
     Let $\nu\in M(K\times B_{E^*})$ be arbitrary. Let $\widetilde{\nu}$ be the measure provided by Proposition~\ref{P:KT*nu konstrukce}. Then $\widetilde{\nu}=WT^*\nu$, where $W$ is the operator from \cite[Proposition 3.3]{batty-vector} (denoted by $K$ in the quoted paper).
\end{cor}

\begin{proof}
    By Proposition~\ref{P:KT*nu konstrukce} and Lemma~\ref{L:KT*nu dukaz}, the measure $\widetilde{\nu}$ has the properties uniquely determining $WT^*\nu$ by \cite[Proposition 3.3]{batty-vector}.
\end{proof}

\begin{remark}
    Our approach provides an alternative construction of the operator $W$ from \cite{batty-vector}.
    The original construction in \cite{batty-vector} uses the assignment $f\mapsto p_f$ from Corollary~\ref{cor:f-p}, which may be extended to a real-linear mapping $\D-\D\to\B-\B$ and then the Hahn-Banach and Riesz representation theorem are invoked. Our construction is different -- we start from $\mu\in M(K,E^*)$, find an arbitrary $\nu\in M(K\times B_{E^*})$ with $T^*\nu=\mu$ and then apply Proposition~\ref{P:KT*nu konstrukce}. If we choose $\nu$ such that $\norm{\nu}=\norm{\mu}$ (which is possible by the Hahn-Banach theorem), the construction is a bit simpler.

    We note that our construction uses more advanced tools (disintegration of measures), but provides stronger results (see Theorem~\ref{t:battymain} below) and, moreover, provides a weak$^*$ Radon-Nikod\'ym derivative of $\mu$ with respect to $\abs{\mu}$. Hence Lemma~\ref{L:KT*nu dukaz} may be viewed as an ultimate
    generalization of \cite[Proposition 2.2]{batty-vector}.
\end{remark}

We continue by a further result promised in the introduction which is the optimal version of the uniqueness statement from \cite[p. 540]{batty-vector}.

\begin{thm}\label{t:battymain}
    The following assertions are equivalent:
    \begin{enumerate}[$(1)$]
        \item $E^*$ is strictly convex.
        \item If $\nu_1,\nu_2\in M_+(K\times B_{E^*})$ are such that
        $T^*\nu_1=T^*\nu_2$ and $\norm{\nu_1}=\norm{\nu_2}=\norm{T^*\nu_1}$, then $\nu_1=\nu_2$.
        \item If $\nu\in M_+(K\times B_{E^*})$ is such that $\norm{T^*\nu}=\norm{\nu}$, then $WT^*\nu=\nu$.
    \end{enumerate}
\end{thm}

\begin{proof}
    $(2)\implies(3)$: This is obvious as, given $\nu$ as in $(3)$, the measures $\nu$ and $WT^*\nu$ satisfy the assumptions of $(2)$.

    $(3)\implies(2)$: This is also obvious, as given $\nu_1$ and $\nu_2$ as in $(2)$, $(3)$ yields
    $\nu_1=WT^*\nu_1=WT^*\nu_2=\nu_2$.

    $(2)\implies (1)$: This is proved in \cite[p. 540]{batty-vector}. Let us recall the easy argument. Assume $E^*$ is not strictly convex. Then there are three distinct points $x^*,x_1^*,x_2^*\in S_{E^*}$ such that $x^*=\frac12(x_1^*+x_2^*)$. Fix $t\in K$ and set
    $$\nu_1=\ep_{(t,x^*)}\mbox{ and }\nu_2=\tfrac{1}{2}(\ep_{(t,x_1^*)}+\ep_{(t,x_2^*)}).$$
    Then $\nu_1,\nu_2$ are positive measures, $\nu_1\ne\nu_2$, $T^*\nu_1=\ep_t\otimes x^*$,
    $$T^*\nu_2(\f)=\int T\f\di\nu_2=\tfrac12(x_1^*(\f(t))+x_2^*(\f(t)))=x^*(\f(t)),$$
    so $T^*\nu_2=\ep_t\otimes x^*$.
    Since $\norm{\ep_t\otimes x^*}=1$, the argument is complete.

    $(1)\implies (3)$: Assume $E^*$ is strictly convex. Let $\nu$ be as in $(3)$. By Proposition~\ref{P:vlastnostih}$(a)$ we deduce $\pi_1(\nu)=\abs{T^*\nu}$. Denote this measure by $\sigma$. Let $(\nu_{t})_{t\in K}$ be a disintegration kernel for $\nu$. By Proposition~\ref{P:vlastnostih}$(b)$ we get $r(\nu_t)\in S_{E^*}$ $\sigma$-almost everywhere. Since $\nu_t$ are probability measures (recall that $\nu\ge0$) and $E^*$ is strictly convex, we deduce that $\nu_t=\ep_{r(\nu_t)}$ $\sigma$-almost everywhere.

   Let $\h(t)=r(\nu_{t})$ for $t\in K$ and let $F_0$, $\g$ and $\widetilde{\nu}$ be as in Proposition~\ref{P:KT*nu konstrukce}. 
   By the choice of $F_0$ we have $\norm{\h(t)|_{F_0}}=1$ $\sigma$-almost everywhere (cf. Proposition~\ref{P:vlastnostih}$(d)$) and hence $\g(t)=\h(t)$ $\sigma$-almost everywhere. It follows from Lemma~\ref{L:dezintegrace} that $\nu$
    satisfies the equality from Proposition~\ref{P:KT*nu konstrukce}$(i)$. By uniqueness of $\widetilde{\nu}$ we conclude that $\nu=\widetilde{\nu}$. Using  Lemma~\ref{L:KT*nu dukaz} we deduce $\nu=WT^*\nu$ and the proof is complete.
\end{proof}

The key new result in the previous theorem is implication $(1)\implies(3)$. Indeed, equivalence $(2)\iff(3)$ is easy, implication $(2)\implies(1)$ follows from the example in \cite[p. 540]{batty-vector}, but implication $(1)\implies(3)$ is new. In \cite[p. 540]{batty-vector} a much weaker version is proved -- it is assumed there that $E$ is separable, reflexive and both $E$ and $E^*$ are strictly convex. Using the technique of disintegration we show that strict convexity of $E^*$ is enough, obtaining thus the optimal result.

\section{Orderings of measures}\label{s:ordering}

In this section we analyze some orderings of measures defined using the cones $\B$ and $\D$. It is inspired by \cite[Section 4]{batty-vector}. Since we focus on the whole cones $\B$ and $\D$ and not their subcones, the situation is in fact different. As we will see below, the ordering defined by $\B$ is trivial -- it is not interesting in itself, but just as the trivial case of possible future considerations. On the other hand, the ordering defined by $\D$ enjoys several interesting 
and perhaps surprising features which we try to understand. 

This section is divided to five subsections. In the first one we collect definitions and easy properties of the orderings, in particular, maximal measures with respect to the cone $\D$ are identified. The second subsection has auxiliary nature and its results are applied in the third subsection where we relate the ordering using the cone $\D$ with the classical Choquet ordering (using the method od disintegration). In the fourth subsection we focus on minimal measures with respect to the cone $\D$ and relate them to the classical maximal measures. In the final subsection we address the question of uniqueness of these minimal measures. 

\subsection{Orderings by the cones $\B$ and $\D$ -- basic facts}\label{ss:ordering-basic}
We start by the trivial case. For $\mu_1,\mu_2\in M(K,E^*)$ it is natural to define
$$\mu_1\prec_{\B}\mu_2  \equiv^{\mathrm{df}} \forall p\in\B\colon p(\mu_1)\le p(\mu_2).$$
However, as the following observation says, this is not very interesting.

\begin{obs}
    Let $\mu_1,\mu_2\in M(K,E^*)$. Then $\mu_1\prec_{\B}\mu_2$ if and only if $\mu_1=\mu_2$.
\end{obs}

\begin{proof}
    The `if part' is obvious. To prove the `only if' part assume $\mu_1\prec_{\B}\mu_2$. Then $p(\mu_1)=p(\mu_2)$ for each $p\in \B\cap(-\B)$. Hence, for each $\f\in C(K,E)$ we have
    $$\Re \int \f\di\mu_1=\Re \int \f\di\mu_2.$$
    It follows that $\mu_1$ and $\mu_2$ define the same linear functional on $C(K,E)$, thus $\mu_1=\mu_2$.
\end{proof}

The previous observation witnesses that the ordering $\prec_{\B}$ is trivial as it reduces to the equality. This is related to the fact that we deal only with the whole space $C(K,E)$ and not with a proper function space $H\subsetneq C(K,E)$. This also corresponds to the triviality of the Choquet ordering for function spaces $H=C(K,\er)$ in the classical setting. However, if we look at possible orderings induced by $\D$, the situation is much more complicated and interesting even in this `trivial' case.
So, let us continue by defining the first possible notion of ordering induced by $\D$.

If $\nu_1,\nu_2\in M_+(K\times B_{E^*})$, we define 
$$\nu_1\prec_{\D}\nu_2  \equiv^{\mathrm{df}} \forall f\in\D\colon \int f\di\nu_1\le\int f\di\nu_2.$$
Basic properties are collected in the following lemma.

\begin{lemma}\label{L:precD} Let $\nu_1,\nu_2\in M_+(K\times B_{E^*})$. Then the following assertions hold.
    \begin{enumerate}[$(a)$]
        \item $\nu_1\prec_{\D}\nu_2$ if and only if
        $$\int (\Re T\f_1\wedge\dots\wedge\Re T\f_n)\di\nu_1\le \int (\Re T\f_1\wedge\dots\wedge\Re T\f_n)\di\nu_2$$
        whenever $\f_1,\dots,\f_n\in C(K,E)$.
        \item If $\nu_1\prec_{\D}\nu_2$, then $T^*\nu_1=T^*\nu_2$.
        \item $\nu\prec_{\D} WT^*\nu$ for any $\nu\in M_+(K\times B_{E^*})$.
    \end{enumerate}
\end{lemma}

\begin{proof} $(a)$: The `only if' part is obvious. The `if part' follows from the formula in Step 1 of the proof of Lemma~\ref{L:KT*nu dukaz} using the monotone convergence theorem for nets.

$(b)$: Assume $\nu_1\prec_{\D}\nu_2$. Then
$\int f\di\nu_1=\int f\di\nu_2$ for each $f\in \D\cap(-\D)$. It means that $\int \Re T\f\di\nu_1=\int \Re T\f\di\nu_2$ for each $\f\in C(K,E)$. It easily follows that $T^*\nu_1=T^*\nu_2$.

$(c)$: This is proved in \cite[Lemma 4.1]{batty-vector} using the formulas from Corollary~\ref{cor:f-p}. We are going to present an alternative proof which shows a relationship to the classical Choquet ordering. By $(a)$ we may restrict to functions of the form
$f=f_1\wedge\dots\wedge f_n$, where $f_j=\Re T\f_j$ for some $\f_j\in C(K,E)$.

Let $(\nu_t)_{t\in K}$ be a disintegration kernel for $\nu$. Since $\nu\ge0$, all measures $\nu_t$ are probabilities. Let $\h(t)=r(\nu_t)$ for $t\in K$. Let $F_0,\g,\widetilde{\nu}$ be as in Proposition~\ref{P:KT*nu konstrukce}. By Corollary~\ref{c:battytrans} we know that $\widetilde{\nu}=WT^*\nu$.
Moreover, let $F\subset E$ be a separable subspace containing $F_0\cup \bigcup_{j=1}^n\f_j(K)$. Then
$$\begin{aligned}
    \int_{K\times B_{E^*}} f\di WT^*\nu&=\int_{K\times B_{E^*}} f\di\widetilde{\nu}= \int_K f(t,\g(t))\di \abs{T^*\nu}(t)\\&=  \int_K f(t,\g(t)) \norm{\h(t)|_{F_0}}\di \pi_1(\nu)(t) \\
    &=  \int_K f(t,\norm{\h(t)|_{F_0}}\g(t))\di \pi_1(\nu)(t) 
    \\&=  \int_K f(t,\h(t))\di \pi_1(\nu)(t) =  \int_K f(t,r(\nu_t))\di \pi_1(\nu)(t) 
    \\&\ge \int_K\left(\int_{B_{H^*}} f(t,x^*)\di\nu_t(x^*)\right)\di\pi_1(\nu)(t)=\int_{K\times B_{E^*}} f\di\nu.
\end{aligned}$$
The first equality follows from Corollary~\ref{c:battytrans}, the second one from 
Proposition~\ref{P:KT*nu konstrukce}$(i)$
(note that $f\in C^0(K\times B_{E^*})$) and  the third one follows from Lemma~\ref{L:hustotavariace}$(b)$. In the fourth equality we use that $f(t,\cdot)$ is superlinear. 

Let us explain the fifth equality. Note that under our assumption we have (due to the choice of $F$)
$$f(t,x^*)=\min_{1\le j\le n} \Re x^*(\f_j(t))= \min_{1\le j\le n} \Re x^*|_F(\f_j(t))$$
for $(t,x^*)\in K\times B_{E^*}$. So, if $\h(t)|_F=0$, then
$$f(t,\h(t))=0=f(t,\norm{\h(t)|_{F_0}}\g(t)).$$
Further, by Proposition~\ref{P:KT*nu konstrukce} we know that $\norm{\h(t)|_{F_0}}\g(t)=\h(t)$ if $\h(t)|_{F_0}\ne0$. Hence,
$$f(t,\h(t))=f(t,\norm{\h(t)|_{F_0}}\g(t))\mbox{ unless }\h(t)|_{F_0}=0 \ \&\ \h(t)|_F\ne0.$$
But this set has $\pi_1(\nu)$-measure zero by Lemma~\ref{L:hustotavariace}$(a)$. 
This completes the proof of the fifth equality.

The sixth equality follows from the choice of $\h$. The inequality follows from the fact that $r(\nu_t)$ is the barycenter of $\nu_t$ and $f(t,\cdot)$ is a continuous concave function on $B_{E^*}$. The last equality follows from Lemma~\ref{L:dezintegrace}.

This completes the proof.
\end{proof}

The relation $\prec_{\D}$ is obviously reflexive and transitive, so it is a pre-order. However, it is not a partial order as the weak antisymmetry fails. Indeed, if $t\in K$ and $x^*\in S_{E^*}$ are arbitrary, then the measures
$\ep_{(t,x^*)}$ and $2\ep_{(t,\frac12x^*)}$ coincide on all functions from $\D$ (recall that such functions are positively homogeneous in the second variable). Therefore, we consider (inspired by \cite{batty-vector}) a finer relation $\prec_{\D,c}$ defined by
$$\nu_1\prec_{\D,c}\nu_2 \equiv^{\mathrm{df}}\nu_1\prec\nu_2\mbox{ and }\norm{\nu_2}\le\norm{\nu_1}.$$
This is again a pre-order, but not a partial order as witnessed by measures
$$\ep_{(t,x^*)}+2\ep_{(s,\frac12x^*)}, \ep_{(s,x^*)}+2\ep_{(t,\frac12x^*)},$$
where $s,t\in K$ are two distinct points and $x^*\in S_{E^*}$. The following proposition summarizes relationship of pre-orders $\prec_{\D}$ and $\prec_{\D,c}$ and identifies $\prec_{\D,c}$-maximal measures with Batty's measures.

\begin{prop}\label{P:precD na sfere} \ 
    \begin{enumerate}[$(a)$]
        \item A measure $\nu\in M_+(K\times B_{E^*})$ is $\prec_{\D,c}$-maximal if and only if $\nu=WT^*\nu$.
        \item Relations $\prec_{\D}$ and $\prec_{\D,c}$ restricted to measures carried by $K\times S_{E^*}$ coincide and are partial orders.
    \end{enumerate}
\end{prop}

\begin{proof} $(b)$: Assume that $\nu_1,\nu_2$ are carried by $K\times S_{E^*}$ and $\nu_1\prec_{\D}\nu_2$. Since the function $f(t,x^*)=-\norm{x^*}$ belongs to $\D$ (this is obvious, as noticed in \cite[Example 2.3(1)]{batty-vector}) and $f=-1$ on $K\times S_{E^*}$, we deduce that $-\norm{\nu_1}\le-\norm{\nu_2}$. Thus $\nu_1\prec_{\D,c}\nu_2$. This proves the coincidence of the two relations.

To prove they are partial orders, it is enough to establish the weak antisymmetry. But this is proved in \cite[Lemma 3.2]{batty-vector}.

$(a)$: Assume that $WT^*\nu_1=\nu_1$ and $\nu_1\prec_{\D,c}\nu_2$. Then $T^*\nu_2=T^*\nu_1$ (by Lemma~\ref{L:precD}$(b)$) and hence
$$\norm{\nu_2}\le\norm{\nu_1}=\norm{T^*\nu_1}=\norm{T^*\nu_2}\le\norm{\nu_2},$$
so $\norm{\nu_2}=\norm{\nu_1}$. Further, by Lemma~\ref{L:precD}$(c)$ we get
$$\nu_2\prec_{\D} WT^*\nu_2=WT^*\nu_1=\nu_1.$$
We conclude $\nu_2\prec_{\D,c}\nu_1$. Thus $\nu_1$ is $\prec_{\D,c}$-maximal.
In fact, as both $\nu_1$ and $\nu_2$ are carried by $K\times S_{E^*}$ (by Lemma~\ref{L:nesena sferou}), by the already proven assertion $(b)$ we get $\nu_2=\nu_1$.

Next assume that $\nu$ is $\prec_{\D,c}$-maximal. Since $\nu\prec_{\D} WT^*\nu$ (by Lemma~\ref{L:precD}) and $\norm{WT^*\nu}\le\norm{\nu}$, we deduce $\nu\prec_{\D,c} WT^*\nu$. By the maximality of $\nu$ we get $WT^*\nu\prec_{\D,c}\nu$. Thus $\norm{\nu}=\norm{T^*\nu}$ and so
$\nu$ is carried by $K\times S_{E^*}$ (by Lemma~\ref{L:nesena sferou}). Therefore, using $(b)$ we deduce $\nu=WT^*\nu$.
    \end{proof}

\subsection{On the Choquet ordering of measures on $B_{E^*}$} 

In this auxiliary subsection we present a result on the  Choquet ordering on probabilities on $B_{E^*}$ for a Banach space $E$. This seems to be interesting in itself, but our main motivation is to apply it to a more detailed analysis of the pre-orders $\prec_{\D}$ and $\prec_{\D,c}$ in Section~\ref{ss:more} below.
The promised result reads as follows. 

\begin{thm}\label{t:ordering-sfera}  Let $\mu,\nu$ be two probability measures on $B_{E^*}$ with the same barycenter. Assume the common barycenter lies on the sphere. If $\int p\di\mu\le \int p\di\nu$ for each weak$^*$ continuous sublinear function $p$ on $E^*$, then $\mu\prec \nu$ in the Choquet ordering.
\end{thm}

Note that a probability on the ball with barycenter on the sphere is necessarily carried by the sphere. We further note that it is enough to prove this theorem for a real Banach space $E$, since the complex case may be deduced by considering the real version of the space. Therefore, in this section we assume that $E$ is a real Banach space. To prove the theorem we need two lemmata.

\begin{lemma}
 \label{l:kompakt}   Let $\mu\in M_1(B_{E^*})$ be a measure with the barycenter on the sphere. 
Then for each $\ep>0$ there exists a weak$^*$ compact convex set $K\subset S_{E^*}$ with $\mu(K)>1-\ep$.
\end{lemma}

\begin{proof} Let $x^*=r(\mu)$. Let $(x_n)$ be a sequence in $B_E$ with $x^*(x_n)\to 1$.
Define $f_n(y^*)=y^*(x_n)$ for $y^*\in B_{E^*}$ and $n\in\en$. Given $n\in\en$, $f_n$ is a continuous affine function on $B_{E^*}$ satisfying $-1\le f_n\le 1$ on $B_{E^*}$. Moreover,
$$\int f_n\di\mu=f_n(x^*)=x^*(x_n)\to1.$$
It follows that $f_n\to 1$ in $L^1(\mu)$, hence, up to passing to a subsequence, we may assume that $f_n\to1$ $\mu$-almost everywhere. Let $F=\{y^*\in B_{E^*}\setsep y^*(x_n)\to 1\}$. Then $F\subset S_{E^*}$, it is a convex set of full measure and it may be expressed as
\[
F=\bigcap_{k\in\en}\bigcup_{n\in\en}F_{k,n},
\]
where
\[
F_{k,n}=\bigcap_{m\ge n}\{y^*\in B_{E^*}\setsep y^*(x_m)\ge 1-\tfrac1k\},\quad k,n\in\en.
\]
Observe that $F_{k,n}$ are weak$^*$ compact convex sets with $F_{k,n}\subset F_{k,n+1}$ for $k,n\in\en$.
Since 
$$1=\mu(F)=\mu\left(\bigcap_{k\in\en}\bigcup_{n\in\en} F_{k,n}\right),$$
we get
$$\mu\left(\bigcup_{n\in\en} F_{k,n}\right)=1\mbox{ for each }k\in\en.$$  Let $n_k\in\en$ be such that $\mu(F_{k,n_k})>1-\frac{\ep}{2^k}$.
Then $K=\bigcap_{k\in\en} F_{n,k_n}$ is a weak$^*$ compact convex set with $\mu(K)>1-\ep$. Further, clearly $K\subset F\subset S_{E^*}$.
\end{proof}

\begin{lemma} \label{l:lipschitz}
Let $E$ be a real Banach space.
Let $f\colon E^*\to \er$ be a weak$^*$ lower semicontinuous $L$-Lipschitz convex function (where $L>0$) with $f(0)=0$. Let $K\subset S_{E^*}$ be a weak$^*$ compact convex set. For each $x^*\in K$ we have
$$f(x^*)=\sup\{ x^*(x)\setsep x\in E, \norm{x}\le 6L, y^*(x)\le f(y^*)\mbox{ for each }y^*\in K\}.$$
\end{lemma}

\begin{proof} The proof is divided to several steps.

\smallskip

{\tt Step 1.} Set $C=\bigcup_{\alpha\ge 0} \alpha K$. Then $C$ is a weak$^*$ closed convex cone.

\smallskip

Clearly, $C$ is a convex cone. Moreover, since $K\subset S_{E^*}$, for each $r>0$ we have
$$C\cap rB_{E^*}=\bigcup_{0\le\alpha\le r} \alpha K,$$
which is weak$^*$ compact, being the image of the compact set $[0,r]\times K$ by the continuous map $(\alpha,x^*)\mapsto \alpha x^*$. We conclude by the Krein-\v{S}mulyan theorem.

\smallskip

{\tt Step 2.} Set $g(\alpha y^*)=\alpha f(y^*)$, $\alpha\in[0,\infty)$ and $y^*\in K$. Then $g$ is a weak$^*$ lower semicontinuous sublinear function on $C$.  

\smallskip

It is clear that $g$ is well defined and positive homogeneous. To prove it is subadditive observe that, given $y^*,z^*\in K$, $\alpha,\beta\ge0$ with $\alpha+\beta>0$, we have
$$\begin{aligned}    
g(\alpha y^*+\beta z^*)&=g((\alpha+\beta)\tfrac{\alpha y^*+\beta z^*}{\alpha+\beta})=(\alpha+\beta)f(\tfrac{\alpha y^*+\beta z^*}{\alpha+\beta})\\&\le(\alpha+\beta)\cdot\tfrac{\alpha f(y^*)+\beta f(z^*)}{\alpha+\beta}=\alpha f(y^*)+\beta f(z^*)=g(\alpha y^*)+g(\beta z^*).
\end{aligned}$$
To prove it is weak$^*$ lower semicontinuous, it is enough to prove that $[g\le d]$ is a weak$^*$ closed subset of $C$ for each $d\in\er$. Since the set $[g\le d]$ is convex, by the Krein-\v{S}mulyan theorem it is enough to prove that $[g\le d]\cap r B_{E^*}$ is weak$^*$ closed for each $r>0$. So, fix $d\in\er$ and $r>0$. Then
$$[g\le d]\cap rB_{E^*}=\{\alpha x^*\setsep x^*\in K, \alpha\in [0,r], \alpha f(x^*)\le d\}$$
is weak$^*$ compact, being the image of the compact set
$$\{(\alpha,x^*)\in[0,r]\times K\setsep \alpha f(x^*)\le d\}$$
under the continuous map $(\alpha,x^*)\mapsto \alpha x^*$.

\smallskip

{\tt Step 3.} Fix $x^*\in K$ and $t_0<f(x^*)$. Since $f(x^*)\ge-L$, we may assume $t_0>-2L$. Set 
\[
A=\{(y^*,t)\in C\times \er\setsep t\ge g(y^*)\}
\]
and
\[
B=\co(\{(x^*,t_0)\}\cup (x^*+B_{E^*})\times \{-5L\}).
\]
Then $A,B$ are disjoint nonempty convex sets in $Z=(E^*,w^*)\times \er$. Moreover, $A$ is closed and $B$ is compact.

\smallskip

Obviously $A$ and $B$ are nonempty and convex and $B$ is compact.  The set $A$ is closed by Steps 1 and 2. It remains to prove they are disjoint. Assume that $(\alpha y^*,t)\in A\cap B$ for some $y^*\in K$, $\alpha\in[0,\infty)$ and $t\in \er$. Then $g(\alpha y^*)=\alpha f(y^*)\le t$ and there is $c\in [0,1]$ and $u^*\in B_E^*$ such that
\begin{equation}\label{eq:prunik}
(\alpha y^*,t)=(x^*+(1-c)u^*,ct_0-5L(1-c)).\end{equation}
We distinguish several cases:

\emph{Case 1.} $c=1$: Then $t=t_0$ and $\alpha y^*=x^*$, hence $\alpha=1$ and $y^*=x^*$. So,
\[
t_0=t\ge f(y^*)=f(x^*)>t_0,
\]
a contradiction.

\emph{Case 2.} $c<1$ and $\alpha\le 1$: Then
we have
\[
f(\alpha y^*)\le \alpha f(y^*)\le t=ct_0-5L(1-c)\le cf(x^*)-5L(1-c).
\]
Further,
\[
f(\alpha y^*)=f(x^*+(1-c)u^*)-f(x^*)+f(x^*)\ge f(x^*)-L\norm{(1-c)u^*}\ge f(x^*)-L(1-c).
\]
Putting together we obtain
\[
f(x^*)-L(1-c)\le cf(x^*)-5L(1-c).
\]
This gives $f(x^*)\le -4L$, a contradiction (recall that $f(x^*)\ge -L$).

\emph{Case 3.} $c<1$ and $\alpha>1$: Then we have
\[
\begin{aligned}
f(\alpha y^*)-\alpha f(y^*)&=f(\alpha y^*)-f(y^*)+f(y^*)-\alpha f(y^*)\\
&\le L\norm{(\alpha-1)y^*}+\abs{f(y^*)}(\alpha-1)\le 2L(\alpha-1).\\
\end{aligned}
\]
Hence 
\[
 f(\alpha y^*)\le\alpha f(y^*)+2L(\alpha-1).
\]
On the other hand,
$$f(\alpha y^*)=f(x^*+(1-c)u^*)\ge f(x^*)-L(1-c).$$
Putting together we obtain
$$\begin{aligned}
    f(x^*)-L(1-c)&\le \alpha f(y^*)+2L(\alpha-1)\le t+2L(\alpha-1) \\&= ct_0-5L(1-c)+2L(\alpha-1)
\\&\le cf(x^*)-5L(1-c)+2L(\alpha-1)\\&\le cf(x^*)-5L(1-c)+2L(1-c),\end{aligned}$$
where the last inequality follows from comparison of the first coordinates in \eqref{eq:prunik} and the triangle inequality. We deduce
$f(x^*)\le-2L$, a contradiction.

\smallskip

{\tt Step 4.} Construction of $x$:

\smallskip

Using the Hahn-Banach separation theorem we find $x\in E$ and $\omega\in\er$   such that
\[
\sup\{y^*(x)+\omega s\setsep (y^*,s)\in B\}< \inf\{z^*(x)+\omega t\setsep (z^*,t)\in A\}.
\]
By the definition of $A$ we see that necessarily $\omega\ge0$. By setting $y^*=z^*=x^*$, $s=t_0$ and $t=f(x^*)$ we deduce that $\omega >0$. Hence, up to scaling  we may assume $\omega=1$, i.e.,
\[
\sup\{y^*(x)+s\setsep (y^*,s)\in B\}< \inf\{z^*(x)+t\setsep (z^*,t)\in A\}.
\]
We note that
$$\begin{aligned}  
\inf\{z^*(x)+t\setsep (z^*,t)\in A\}&=\inf\{z^*(x)+g(z^*)\setsep z^*\in C\}\\&=\inf\{t(z^*(x)+f(z^*))\setsep z^*\in K, t\ge0\},\end{aligned}$$
which is obviously either $0$ or $-\infty$. But the second possibility cannot take place, so the infimum is $0$. In particular,
$$z^*(-x)\le f(z^*) \mbox{ for }z^*\in K.$$
Further, since $(x^*,t_0)\in B$, we deduce that $x^*(x)+t_0<0$, i.e., $x^*(-x)>t_0$. Finally,
$$0>\sup\{ x^*(x)+y^*(x)-5L\setsep y^*\in B_{E^*}\}=x^*(x)+\norm{x}-5L,$$
so
$$\norm{-x}=\norm{x}<5L-x^*(x)\le 5L+f(x^*)\le 6L.$$
This completes the proof.
\end{proof}

Now were are ready to prove the theorem:

\begin{proof}[Proof of Theorem~\ref{t:ordering-sfera}]
We proceed by contraposition. Assume that $\mu\not\prec\nu$ in the Choquet ordering. By \cite[Proposition 3.56]{lmns} there are weak$^*$ continuous affine functions $f_1,\dots,f_n$ on $B_{E^*}$ such that
$$\int \max\{f_1,\dots,f_n\}\di\nu<\int \max\{f_1,\dots,f_n\}\di\mu.$$
Since any weak$^*$ continuous affine function on $B_{E^*}$ is Lipschitz (it is a function of the form $x^*\mapsto x^*(x)+c$ for some $x\in E$ and $c\in\er$), we have a Lipschitz weak$^*$ continuous convex function $f:E^*\to\er$ with $\int f\di \nu<\int f\di\mu$. Since both $\mu$ and $\nu$ are probabilities, up to replacing $f$ by $f-f(0)$ we may assume $f(0)=0$. Let $L$ denote the Lipschitz constant of $f$. Clearly $L>0$.

Fix $\ep>0$. Since $\mu$ and $\nu$ have the same barycenter, the measure $\frac12(\mu+\nu)$ has the barycenter on the sphere. Therefore we may apply Lemma~\ref{l:kompakt} to find a weak$^*$ compact convex set $K\subset S_{E^*}$ such that $(\mu+\nu)(B_{E^*}\setminus K)<\ep$. By Lemma~\ref{l:lipschitz} we have for $x^*\in K$ 
$$
\begin{aligned}
f(x^*)&=\sup\{g(x^*)\setsep g:E^*\to\er \mbox{ weak$^*$ continuous, linear,} \\&\qquad\qquad 6L\mbox{-Lipschitz}, g\le f\mbox{ on }K\} 
\\&=\sup\{g(x^*)\setsep g:E^*\to\er \mbox{ weak$^*$ continuous, sublinear,} \\&\qquad\qquad6L\mbox{-Lipschitz}, g\le f\mbox{ on }K\}. 
\end{aligned}
$$ 
Indeed, the first equality follows directly from Lemma~\ref{l:lipschitz} and the second one is a trivial consequence. Since the family of functions from the last expression is upwards directed, the monotone convergence theorem for nets 
provides such $g$ with $\int_K g\di\mu>\int_K f\di\mu-\ep$.
Then
$$\begin{aligned}
\int g\di\mu&\ge \int_K g\di\mu-6L\ep >\int_K f\di\mu-\ep-6L\ep\ge \int f\di\mu-\ep-7L\ep,
\end{aligned}$$
where we used the choice of $K$, the choice of $g$, equalities $f(0)=g(0)=0$ and the assumptions that $f$ is $L$-Lipschitz and $g$ is $6L$-Lipschitz.
On the other hand, similarly we get.
$$\int g\di\nu\le \int_K g\di\nu+6L\ep \le\int_K f\di\nu+6L\ep\le \int f\di\nu+7L\ep.$$
It is now clear that, choosing $\ep>0$ small enough we may achieve $\int g\di\nu<\int g\di\mu$.
This completes the proof.   
\end{proof}

\subsection{More on orderings defined by the cone $\D$}\label{ss:more}

We now analyze in more detail the pre-orders $\prec_{\D}$ and $\prec_{\D,c}$ and their relationship to the classical Choquet order. To this end we will use the result from previous subsection and the technique of disintegration of measures. We restrict ourselves to positive measures of the minimal norm, i.e., to the set
$$N=\{\nu\in M_+(K\times B_{E^*})\setsep \norm{T^*\nu}=\norm{\nu}\}.$$
Further, given $\mu\in M(K,E^*)$, we set
$$N(\mu)=\{\nu\in N\setsep T^*\nu=\mu\}.$$
We start by collecting a few basic facts on the set $N$:

\begin{obs}\label{obs:N} \ 
\begin{enumerate}[$(a)$]
    \item Pre-orders $\prec_{\D}$ and $\prec_{\D,c}$ coincide on $N$.
    \item The relation $\prec_{\D}$ restricted to $N$ is a partial order.
    \item If $\nu_1,\nu_2\in N$ are such that $\nu_1\prec_{\D}\nu_2$, then these two measures belong to the same $N(\mu)$.
\end{enumerate}  
\end{obs}

\begin{proof}
    By Lemma~\ref{L:nesena sferou} the measures from $N$ are carried by $K\times S_{E^*}$. Assertions $(a)$ and $(b)$ thus follow from Proposition~\ref{P:precD na sfere}$(b)$. Assertion $(c)$ follows from Lemma~\ref{L:precD}$(b)$.
\end{proof}

In order to address the case of possibly nonseparable $E$ we will use restriction maps to separable spaces. More specifically, if $F\subset E$ is a (separable) subspace, let $R_F:E^*\to F^*$ be the canonical restriction map. Then $R_F$ restricted to $B_{E^*}$ is a continuous surjection of $B_{E^*}$ onto $B_{F^*}$.

\begin{lemma}\label{L:rovnost tezist}
    Let $\mu\in M(K,E^*)\setminus\{0\}$ be given. For $\nu\in N(\mu)$ let $(\nu_t)_{t\in K}$ be a disintegration kernel of $\nu$. Then the following hold:
    \begin{enumerate}[$(a)$]
        \item If $\nu\in N(\mu)$, then $r(\nu_t)\in S_{E^*}$ for $\abs{\mu}$-almost all $t\in K$.
        \item Let $\nu\in N(\mu)$. If $F\subset E$ is a sufficently large separable subspace of $E$, then $\norm{R_F\circ\mu}=\norm{\mu}$ and $(\id\times R_F)(\nu)\in N(R_F\circ\mu)$.
        \item Let $\nu_1,\nu_2\in N(\mu)$. If $F\subset E$ is a sufficently large separable subspace of $E$, then $r(R_F(\nu_{1,t}))=r(R_F(\nu_{2,t}))\in S_{F^*}$  for $\abs{\mu}$-almost all $t\in K$.
    \end{enumerate}
\end{lemma}

\begin{proof}
    $(a)$: This follows from Proposition~\ref{P:vlastnostih}$(b)$.

    $(b)$: Let $F\subset E$ be an arbitrary separable subspace. Let $T_F:C(K,F)\to C(K\times B_{F^*})$ be the respective variant of the operator $T$. If $A\subset K$ is Borel and $x\in F$, then we get by \eqref{eq:hustad} 
    $$\begin{aligned}
       T_F^*((\id\times R_F)(\nu))(A)(x)&=\int_{A\times B_{F^*}} y^*(x)\di (\id\times R_F)(\nu)(t,y^*)
       \\&=\int_{A\times B_{E^*}} (R_Fx^*)(x)\di\nu(t,x^*)=\int_{A\times B_{E^*}} x^*(x)\di\nu(t,x^*)
       \\&=\mu(A)(x)=(R_F\circ\mu)(A)(x).
    \end{aligned}$$
    We deduce that $T_F^*((\id\times R_F)(\nu))=R_F\circ\mu$. Since clearly $\norm{(\id\times R_F)(\nu)}=\norm{\nu}$ (as $\nu\ge0$), it is enough to take $F$ so large that $\norm{R_F\circ \mu}=\norm{\mu}$. This may be achieved easily -- similarly as in the proof of Proposition~\ref{P:vlastnostih}$(d)$ we find a sequence $(\f_n)$ in $C(K,E)$ such that
    $$\norm{\mu}=\sup_{n\in\en} \abs{\int \f_n\di\mu}$$
    and let $F$ be the closed linear span of $\bigcup_n\f_n(K)$. (Any larger $F$ works as well.)

    $(c)$: Let $F$ be a subspace provided by $(b)$ which works simultaneously for $\nu_1$ and $\nu_2$. Let $\h_1$ and $\h_2$ be the functions provided by Proposition~\ref{P:hustotaT*nu} for $\nu_1$ and $\nu_2$.
    For $j=1,2$ set $\widetilde{\nu_j}=(\id\times R_F)(\nu_j)$. Then for each $\f\in C(K,F)$ we have
    $$\begin{aligned}
         \int_K \f\di T_F^*(\widetilde{\nu_j})&=\int_{K\times B_{F^*}} T_F\f \di\widetilde{\nu_j}
    =\int_{K\times B_{E^*}} (T_F\f)\circ (\id\times R_F) \di\nu_j 
    \\&= \int_{K\times B_{E^*}} R_F(x^*)(\f(t))\di\nu_j(t,x^*)= \int_{K\times B_{E^*}}  x^*(\f(t))\di\nu_j(t,x^*)
    \\&= \int_K \h_j(t)(\f(t))\di\abs{\mu}(t)= \int_K \h_j(t)|_F(\f(t))\di\abs{\mu}(t).
       \end{aligned}$$
   The first equality follow follows from the definition of $T_F^*$ and the second one follows from the rules of integration with respect to the image of a measure. The third one follows from the definition of $T_F$. The fourth one is obvious (as $\f(t)\in F$
   for each $t\in K$). The fifth one follows from the choice of $\h_j$ and the last one is again obvious. Hence the function $R_F\circ \h_j$ is a possible choice of $\h$ associated to $\widetilde{\nu_j}$ by Proposition~\ref{P:hustotaT*nu}. Using the choice of $F$ and 
   combining Lemma~\ref{L:dezintegrace kvocient} and Lemma~\ref{L:hustota jednoznacnost} we deduce that
   $$r(\widetilde{\nu_{1,t}})=\h_1(t)|_F=\h_2(t)|_F=r(\widetilde{\nu_{2,t}})\quad\abs{\mu}\mbox{-almost everywhere}.$$
    This completes the proof. 
\end{proof}

We continue by collecting several results on separable factorization which will be useful in the sequel.

\begin{lemma}
\label{l:separ-faktor}
\begin{enumerate}[$(a)$]
    \item Let $f\in C(B_{E^*})$ be given. Then there is a separable subspace $F\subset E$ and  $g\in C(B_{F^*})$ such that $f=g\circ R_F$.
\item  Let $S\subset C(B_{E^*})$ be a norm-separable set. Then there exists a separable subspace $F\subset E$ such that for each $f\in S$ there exists $g\in C(B_{F^*})$ with $f=g\circ R_F$.
    \item Let $f\in C(K\times B_{E^*})$ be given. Then there exists a separable space $F\subset E$ and  $g\in C(K\times B_{F^*})$ such that $f=g\circ (\operatorname{id}\times R_F)$.
\end{enumerate}
\end{lemma}

\begin{proof}
$(a)$: Let $Y$ denote the set of all $f\in C(B_{E^*})$ admitting the required factorization. It is clear that $Y$ is a closed subalgebra containing constants and stable to complex conjugation. Moreover, for any $x\in E$ the function $x^*\mapsto x^*(x)$ belongs to $Y$ (it is enough to take $F=\span\{x\}$). Hence $Y$ separates points of $B_{E^*}$, so we conclude by the Stone-Weierstrass theorem.

$(b)$: Let $D\subset S$ be a countable dense subset of $S$. For each $d\in D$ we select using $(a)$ a separable space $F_d\subset E$ such that there exists $g_d\in C(B_{{F_d}^*})$ satisfying $f=g_d\circ R_{F_d}$. Then the space $F=\overline{\span \bigcup_{d\in D} F_d}$ is separable and any function from $D$ can be factorized as a function from $C(B_{F^*})$. Now it is easy to see that any function $f$ from $S\subset \overline{D}^{\norm{\cdot}}$ can be factorized as $f=g_f\circ R_F$ for some $g_f\in C(B_{F^*})$.

$(c)$: We proceed similarly as in $(a)$. Let $Z$ denote the set of all $f\in C(K\times B_{E^*})$ admitting the required factorization. It is clear that $Z$ is a closed subalgebra containing constants and stable to complex conjugation. Moreover, by $(a)$ $Z$ contains all functions of the form
$$(t,x^*)\mapsto f(x^*)\mbox{ where }f\in C(B_{E^*}).$$
Obviously it contains also all functions of the form
$$(t,x^*)\mapsto f(t)\mbox{ where }f\in C(K).$$
It follows that $Z$ separates points of $B_{E^*}$, so we conclude by the Stone-Weierstrass theorem.  
\end{proof}

Using the previous lemma we may easily show that the relation $\prec_{\D}$ is separably determined.
It is the content of the following proposition.

\begin{prop}\label{P:precD-sepred}
   Let $\nu_1,\nu_2\in M_+(K\times B_{E^*})$ Then the following assertions are equivalent:
         \begin{enumerate}[$(1)$]
            \item $\nu_1\prec_{\D}\nu_2$;
            \item $(\id\times R_F)(\nu_1)\prec_{\D}(\id\times R_F)(\nu_2)$ for each separable subspace $F\subset E$;
            \item $(\id\times R_F)(\nu_1)\prec_{\D}(\id\times R_F)(\nu_2)$ for each $F$ from a cofinal family of separable subspaces of $E$.
        \end{enumerate}   
\end{prop}

\begin{proof} To clarify the meaning of $(3)$ let us recall that a family of separable subspaces of $E$ is \emph{cofinal} if any separable subspace of $E$ is contained in a member of the family.

       We now proceed with the proof itself.
        Implication $(1)\implies(2)$ is easy: If $f$ belongs to the cone $\D$ on $K\times F^*$, then $f\circ(\id\times R_F)$ belongs to the cone $\D$ on $K\times E^*$ and we may use the rules of integration with respect to the image of a measure.

      Implication $(2)\implies (3)$ is trivial.

    $(3)\implies(1)$: Assume assertion $(3)$ holds. To prove that $\nu_1\prec_{\D}\nu_2$ it is enough to show that $\int f\di\nu_1\le \int f\di\nu_2$ for any continuous $f$ from $\D$ (by Lemma~\ref{L:precD}$(a)$). So, fix  such $f$. By Lemma~\ref{l:separ-faktor}$(c)$ there is a separable subspace $F\subset E$ and $g\in C(K\times B_{F^*})$ such that $f=g\circ(\id\times R_F)$.
    By $(3)$ we may assume, up to enlarging $F$, that $(\id\times R_F)(\nu_1)\prec_{\D}(\id\times R_F)(\nu_2)$.  Since $g$ clearly belongs to the cone $\D$ on $K\times B_{F^*}$, we get  $\int g\di (\id\times R_F)(\nu_1)\le \int g\di(\id\times R_F)(\nu_2)$. Then we conclude using the rules of integration with respect to the image of a measure.
\end{proof}

Now we are going to characterize the relation $\prec_{\D}$ on $N$ and to relate it to the classical Choquet ordering and to the Choquet theory of cones. To this end we set
$$\K=\{f\in C(K\times B_{E^*})\setsep f(t,\cdot)\mbox{ is concave for each }t\in K\}.$$
And, naturally, for $\nu_1,\nu_2\in N$ we set
$$\nu_1\prec_{\K} \nu_2 \equiv^{\mbox{df}} \forall f\in \K \colon \int f\di\nu_1\le \int f\di\nu_2.$$
The promised characterization is the content of the following theorem.

\begin{thm}\label{T:precDnaN(mu)}
    Let $\nu_1,\nu_2\in N(\mu)$ be given. Let $(\nu_{1,t})_{t\in K}$ and $(\nu_{2,t})_{t\in K}$ be their disintegration kernels. Then the following assertions are equivalent.
    \begin{enumerate}[$(1)$]
        \item $\nu_1\prec_{\D}\nu_2$.
        \item If $p:B_{E^*}\to \er$ is weak$^*$ continuous and sublinear, then
        $$\int p\di\nu_{2,t}\le\int p\di\nu_{1,t}\mbox{ for $\abs{\mu}$-almost all }t\in K.$$
        \item If $g:B_{E^*}\to \er$ is weak$^*$ continuous and convex, then
        $$\int g\di\nu_{2,t}\le\int g\di\nu_{1,t}\mbox{ for $\abs{\mu}$-almost all }t\in K.$$
        \item For each $f\in\K$ we have
        $$\int f(t,x^*)\di\nu_{1,t}(x^*)\le\int f(t,x^*)\di\nu_{2,t}(x^*)\mbox{ for $\abs{\mu}$-almost all }t\in K.$$
        \item $\nu_1\prec_{\K}\nu_2$.
    \end{enumerate}
\end{thm}

\begin{proof}
    $(1)\implies(2)$: We proceed by contraposition. Assume $(2)$ fails and fix some $p$ witnessing it. Then $p$ is bounded on $B_{E^*}$, say, $\abs{p}\le C$ on $B_{E^*}$. Further,
$$\left\{t\in K\setsep \int p\di\nu_{2,t}> \int p\di\nu_{1,t}\right\}$$
is not a $\abs{\mu}$-measure zero set. By $\sigma$-additivity there is some $\delta>0$ such that
the set
$$\left\{t\in K\setsep \int p\di\nu_{2,t}>\delta+ \int p\di\nu_{1,t}\right\}$$
is not a $\abs{\mu}$-measure zero set. But this set is measurable, so there is a compact set $F\subset K$ with $\abs{\mu}(F)>0$ such that 
$$\int p\di\nu_{2,t}>\delta+ \int p\di\nu_{1,t}\quad\mbox{for }t\in F.$$
Let $\ep>0$ be arbitrary. By the regularity of $\abs{\mu}$ we find an open subset $U\subset K$ containing $F$ such that $\abs{\mu}(U\setminus F)<\ep$. Fix a continuous function $g:K\to [0,1]$ such that $g=1$ on $F$ and $g=0$ on $K\setminus U$. Then
$$f(t,x^*)=-g(t)p(x^*),\qquad (t,x^*)\in K\times B_{E^*},$$
is a continuous function from $\D$. Moreover,
$$\begin{aligned}
    \int_{K\times B_{E^*}} f\di\nu_2& = \int_K\left(\int_{B_{E^*}} -g(t)p(x^*)\di\nu_{2,t}(x^*)\right)\di\abs{\mu}(t)
    \\&=-\int_F\left(\int_{B_{E^*}} p(x^*)\di\nu_{2,t}(x^*)\right)\di\abs{\mu}(t)\\&\qquad\qquad-\int_{U\setminus F}\left(\int_{B_{E^*}}g(t)p(x^*)\di\nu_{2,t}(x^*)\right)\di\abs{\mu}(t)
    \\&\le -\int_F\left(\delta+\int_{B_{E^*}} p(x^*)\di\nu_{1,t}(x^*)\right)\di\abs{\mu}(t)+C\ep
    \\&= C\ep-\delta\abs{\mu(F)} -\int_K\left(\int_{B_{E^*}} g(t)p(x^*)\di\nu_{1,t}(x^*)\right)\di\abs{\mu}(t)\\
    &\qquad\qquad+\int_{U\setminus F}\left(\int_{B_{E^*}}g(t)p(x^*)\di\nu_{1,t}(x^*)\right)\di\abs{\mu}(t)
    \\&\le 2C\ep-\delta\abs{\mu(F)}+\int_{K\times B_{E^*}} f\di\nu_1.
\end{aligned}$$
Let us explain this computation: The first equality follows from the definition of $f$ and from the properties of disintegration kernels.
The second one follows from the choice of $g$ -- note that $g=1$ on $F$ and $g=0$ on $K\setminus U$. The following inequality follows from the choice of $F$ and $U$ using the estimate $\abs{gp}\le C$. The next equality follows by algebraic manipulation using the fact that $g=1$ on $F$. The final inequality follows from the estimates $\abs{gp}\le C$ and $\abs{\mu}(U\setminus F)<\ep$ together with the definition of $f$ and the properties of disintegration kernels.

If $\ep>0$ is sufficiently small, the above computation yields $\int_{K\times B_{E^*}}f\di\nu_2<\int_{K\times B_{E^*}}f\di\nu_1$, so $\nu_1\not\prec_{\D}\nu_2$, hence $(1)$ fails.

    $(2)\implies(3)$: Let $g:B_{E^*}\to\er$ be weak$^*$ convex and continuous. Let $F\subset E$ be a separable subspace such that $g=g_F\circ R_F$ for some $g_F\in C(B_{E^*})$ (it exists by Lemma~\ref{l:separ-faktor}$(a)$). Up to enlarging $F$ we may assume that the condition from Lemma~\ref{L:rovnost tezist}$(b)$ is fulfilled. Given $p:B_{F^*}\to \er$ weak$^*$ continuous sublinear, we have 
    $$\int p\di R_F(\nu_{2,t})=\int p\circ R_F\di\nu_{2,t} \le \int p\circ R_F\di\nu_{1,t} = \int p\di R_F(\nu_{1,t})$$
    for $\abs{\mu}$-almost all $t\in K$, where we used assumption $(2)$ applied to $p\circ R_F$.
    Since $F$ is separable, the dual ball $(B_{F^*},w^*)$ is metrizable and hence $C(B_{F^*})$ is separable. It now easily follows that for $\abs{\mu}$-almost all $t\in K$ we have
    $$\int p\di R_F(\nu_{2,t}) \le \int p\di R_F(\nu_{1,t})\mbox{ for each $p:B_{E^*}\to\er$ weak$^*$ continuous sublinear}.$$
    Using the validity of the condition from Lemma~\ref{L:rovnost tezist}$(b)$ we deduce from Theorem~\ref{t:ordering-sfera}  that for $\abs{\mu}$-almost all $t\in K$ we have
    $$\int g\di \nu_{2,t}=\int g_F \di R_F(\nu_{2,t}) \le \int g_F\di R_F(\nu_{1,t}) = \int g\di \nu_{1,t},$$
    which completes the argument.

    $(3)\implies(4)$: We proceed by contraposition. Let $f\in \K$ be such that the converse inequality holds on a set of positive $\abs{\mu}$-measure. Using $\sigma$-additivity and regularity we find a compact set $L\subset K$ with $\abs{\mu}(L)>0$ and $\delta>0$ such that
    $$\int f(t,x^*)\di\nu_{1,t}(x^*)>2\delta+\int f(t,x^*)\di\nu_{2,t}(x^*)\mbox{ for }t\in L.$$
    For $t\in K$ let $f_t=f(t,\cdot)$. Then each $f_t$ is a continuous concave function on $B_{E^*}$ and, moreover, the assignment $t\mapsto f_t$ is continuous (from $K$ to $C(B_{E^*})$). Therefore,
    there is a finite set $F\subset L$ such that $\{f_t\setsep t\in F\}$ forms a $\delta$-net of $\{f_t\setsep t\in L\}$. For each $t\in F$ let
    $$L_t=\{s\in L\setsep \norm{f_s-f_t}<\delta\}.$$
    These sets form a finite cover of $L$ by relatively open sets, so at least one of them has positive measure. Hence, fix $t\in F$ with $\abs{\mu}(L_t)>0$. Let $s\in L_t$.
    Then 
    $$\int f_t\di\nu_{1,s}\ge \int  f_s\di\nu_{1,s} -\delta >  \int  f_s\di\nu_{2,s} + \delta 
    \ge \int  f_t\di\nu_{2,s}.$$
    So, the function $f_t$ witnesses that $(3)$ is violated.

    $(4)\implies(5)$: This follows from the disintegration formula (see Lemma~\ref{L:dezintegrace}).

    $(5)\implies(1)$: This is trivial (using Lemma~\ref{L:precD}$(a)$), as continuous functions from $\D$ belong to $\K$.
\end{proof}

For separable $E$ we have the following improvement.

\begin{cor}\label{cor:precD-separable}
    Assume  that $E$ is separable. Let $\nu_1,\nu_2\in N(\mu)$ be given. Let $(\nu_{1,t})_{t\in K}$ and $(\nu_{2,t})_{t\in K}$ be their disintegration kernels. Then the following assertions are equivalent.
    \begin{enumerate}[$(1)$]
        \item $\nu_1\prec_{\D}\nu_2$;
        \item $\nu_{2,t}\prec\nu_{1,t}$ (in the Choquet order on $M_1(B_{E^*})$) for $\abs{\mu}$-almost all $t\in K$.
    \end{enumerate}
\end{cor}

\begin{proof}
    $(2)\implies(1)$: This follows immediately from implication $(3)\implies(1)$ of Theorem~\ref{T:precDnaN(mu)} (separability is not needed). 
    
    $(1)\implies(2)$: Assume $\nu_1\prec_{\D}\nu_2$. Then condition $(3)$ if Theorem~\ref{T:precDnaN(mu)} is valid. Since $E$ is separable, $(B_{E^*},w^*)$ is metrizable and hence $C(B_{E^*})$ is also separable.
    In particular, the cone of weak$^*$ convex continuous functions on $B_{E^*}$ is separable in the sup-norm. Let $C$ be a countable dense subset of this cone. Condition $(3)$ of Theorem~\ref{T:precDnaN(mu)} then yields that for $\abs{\mu}$-almost all $t\in K$ we have
    $$\forall g\in C\colon \int g\di\nu_{2,t}\le\int g\di\nu_{1,t}.$$
    This clearly passes to the closure, hence $C$ may be replaced by the cone of all weak$^*$-continuous convex functions. But this means that condition $(2)$ is fulfilled.
    \end{proof}

Even for nonseparable $E$ we have an analogue of the previous corollary. However, the disintegration kernels have to be chosen in a proper way. It is the content of the following theorem.

\begin{thm}\label{t:precDbodove}
    Let $\mu\in M(K,E^*)\setminus\{0\}$ be given. Then there is a choice of disintegration kernels
    $$\nu\in N(\mu)\mapsto (\nu_t)_{t\in K}$$
    such that for each pair $\nu_1,\nu_2\in N(\mu)$ we have
    $$\nu_1\prec_{\D}\nu_2 \iff \forall t\in K\colon \nu_{2,t}\prec \nu_{1,t}.$$
\end{thm}

\begin{proof}
    Implication `$\impliedby$' holds for any choice of disintegration kernels due to implication $(3)\implies(1)$ from Theorem~\ref{T:precDnaN(mu)}. 
    
    For the converse observe that $\pi_1(\nu)=\abs{\mu}$ for any $\nu\in N(\mu)$ by Proposition~\ref{P:vlastnostih}$(a)$. Therefore we may apply Proposition~\ref{P:vsude} 
    to choose an assignment of disintegration kernels. If $\nu_1,\nu_2\in N(\mu)$ satisfy $\nu_1\prec_{\D}\nu_2$, we conclude by combining Theorem~\ref{T:precDnaN(mu)} (implication $(1)\implies(3)$) with Proposition~\ref{P:vsude}.
\end{proof}

\subsection{On $\prec_{\D}$-minimal measures}

In this section we characterize $\prec_{\D}$-minimal measures in $N$. We start by the separable case which follows easily from Corollary~\ref{cor:precD-separable}.

\begin{prop}\label{P:minD-separable}
Assume that $E$ is separable. Let $\mu\in M(K,E^*)$ be fixed.  Let $\nu\in N(\mu)$ be given and let $(\nu_t)_{t\in K}$ be a disintegration kernel of $\nu$. Then the following assertions are equivalent:
\begin{enumerate}[$(1)$]
    \item $\nu$ is $\prec_{\D}$-minimal in $N(\mu)$;
    \item $\nu_t$ is a maximal measure for $\abs{\mu}$-almost all $t\in K$;
    \item $\nu$ is carried by $K\times\ext B_{E^*}$.
\end{enumerate}  
\end{prop}

\begin{proof} $(2)\implies(1)$: Assume that $\nu_t$ is a maximal measure for $\abs{\mu}$-almost all $t\in K$. Let 
$\nu^\prime\in N$ be such that $\nu^\prime\prec_{\D}\nu$. Let $(\nu^\prime_t)_{t\in K}$ be a disintegration kernel of $\nu^\prime$. By Corollary~\ref{cor:precD-separable} we deduce that 
$\nu_t\prec\nu^\prime_t$ for  $\abs{\mu}$-almost all $t\in K$. By the assumption of maximality we conclude that $\nu_t=\nu^\prime_t$ for  $\abs{\mu}$-almost all $t\in K$, so $\nu^\prime=\nu$. Therefore $\nu$ is $\prec_{\D}$-minimal.

$(1)\implies(2)$: Assume that $\nu$ is $\prec_{\D}$-minimal. Since $B_{E^*}$ is metrizable, Lemma~\ref{L:selekce max} provides a Borel mapping $\Psi:M_1(B_{E^*})\to M_1(B_{E^*})$ such that, for each $\sigma\in M_1(B_{E^*})$, $\Psi(\sigma)$ is a maximal measure such that $\sigma\prec\Psi(\sigma)$. 
For $t\in K$ set $\nu^\prime_t=\Psi(\nu_t)$. Since the assignment $t\mapsto \nu_t$ is measurable (by Lemma~\ref{L:dezintegrace metriz}$(b)$ due to the metrizability of $B_{E^*}$), the assignment $t\mapsto\nu^\prime_t$ is measurable as well. Therefore we may define a measure $\nu^\prime$ on $K\times B_{E^*}$ by
$$\int_{K\times B_{E^*}}f\di\nu^\prime = \int_K\left(\int_{B_{E^*}} f(t,x^*)\di\nu^\prime_t(x^*)\right)\di\abs{\mu}(t),\quad f\in C(K\times B_{E^*}).$$
Since the barycenter of $\nu^\prime_t$ coincide with the barycenter of $\nu_t$, we get that $\nu^\prime\in N(\mu)$ (just use the above formula for $T\f$, $f\in C(K,E)$) and by Lemma~\ref{L:dezintegrace metriz}$(a)$ we deduce that the family $(\nu^\prime_t)_{t\in K}$ is a disintegration kernel of $\nu^\prime$.

Moreover, by Corollary~\ref{cor:precD-separable} we see that $\nu^\prime\prec_{\D}\nu$ and hence, by minimality we get $\nu^\prime=\nu$. Finally, Lemma~\ref{L:dezintegrace metriz}$(c)$ shows that $\nu^\prime_t=\nu_t$ for $\abs{\mu}$-almost all $t\in K$. This completes the proof. 

$(2)\iff(3)$: Assume $\nu\in N(\mu)$. Since $\ext B_{E^*}$ is a $G_\delta$-subset of $B_{E^*}$, by Lemma~\ref{L:dezintegrace} we get
$$\nu(K\times (B_{E^*}\setminus \ext B_{E^*}))=\int_K \nu_t(B_{E^*}\setminus \ext B_{E^*})\di\abs{\mu}(t).$$
So, 
$$\begin{aligned}
\nu\mbox{ is carried by }K\times \ext B_{E^*} &\iff \nu_t(B_{E^*}\setminus \ext B_{E^*})=0\ \abs{\mu}\mbox{-a.e.} \\&
\iff \nu_t\mbox{ is maximal }\abs{\mu}\mbox{-a.e.}\end{aligned}$$ 
\end{proof}

To characterize $\prec_{\D}$-minimal measures in the general case (for possibly nonseparable $E$) we will use the Choquet theory of cones. Indeed, relation $\prec_{\D}$ coincides with $\prec_{\K}$ and $\K$ is clearly a min-stable convex cone containing constants and separating points. Therefore we define (following \cite[Section I.5]{alfsen}) the upper and lower envelopes of a function $g\in C(K\times B_{E^*},\er)$ by
\[
\widehat{g}=\inf\{k\in \K\setsep  k\ge g\},\quad 
\widecheck{g}=\sup\{k\setsep -k\in\K, k\le g\}.
\]
The standard upper and lower envelopes on compact convex sets are denoted by $g^*$ and $g_*$, see~\eqref{eq:obalky}. 
The promised characterization of $\prec_{\D}$-minimal measures is contained in the following theorem.

\begin{thm}
    \label{t:charakterminim}
For $\nu\in N(\mu)$ with a disintegration kernel $(\nu_t)$ the following assertions are equivalent.
\begin{enumerate}[$(1)$]
\item $\nu$ is $\prec_{\D}$-minimal.
\item $\nu$ is $\prec_{\K}$-minimal.
\item $\int f\di\nu=\int\widehat{f}\di\nu$ for each $f\in -\K$.
\item For each $g$ convex weak$^*$ continuous on $B_{E^*}$ we have $\int g\di\nu_t=\int g^*\di\nu_t$ for $\abs{\mu}$-almost all $t$.
\end{enumerate}
\end{thm}

\begin{proof}[Proof of the equivalence of conditions $(1)-(3)$]
Equivalence $(1)\iff(2)$ follows from Theorem~\ref{T:precDnaN(mu)}.

$(2)\iff(3)$: Since $\K$ contains constants, $\nu$ is $\prec_{\K}$-minimal within $N(\mu)$ if and only if it is $\prec_{\K}$-minimal within $M_+(K\times B_{E^*})$. Therefore the equivalence
follows from \cite[Proposition I.5.9]{alfsen}.
\end{proof}

The remaining equivalence requires some auxiliary results contained in the following lemmata.
For $(t,x^*)\in K\times B_{E^*}$ we set (following \cite[p. 46]{alfsen})
$$M_{(t,x^*)}^+(\K)=\left\{\nu\in M_+(K\times B_{E^*})\setsep k(t,x^*)\ge \int k\di\nu\mbox{ for each }k\in \K\right\}.$$

\begin{lemma}
\label{l:repremiry} If $(t,x^*)\in K\times B_{E^*}$, then
$$M_{(t,x^*)}^+(\K)=\left\{\ep_t\times \lambda\setsep \lambda\in M_1(B_{E^*}), r(\lambda)=x^*\right\}.$$    
\end{lemma}

\begin{proof} Inclusion `$\supset$' is obvious. To prove the converse inclusion fix
 $\nu\in M_{(t,x^*)}^+(\K)$.  For any  $h\in C(K,\er)$ the function $h\otimes 1:(t,x^*)\mapsto h(t)$ belongs to $-\K\cap \K$, and thus
\[
\int h \di\pi_1(\nu)=\int (h\circ \pi_1)\di\nu=\int (h\otimes 1)\di\nu=(h\otimes 1)(t,x^*)=h(t).
\]
Therefore $\pi_1(\nu)=\ep_t$. Hence $\nu=\ep_t\times\lambda$ where $\lambda=\pi_2(\nu)$. If $f:B_{E^*}\to\er$ is a weak$^*$ continuous affine function, then the function $1\otimes f:(t,x^*)\mapsto f(x^*)$ belongs to $-\K\cap\K$, hence
$$\int f\di\lambda=\int 1\otimes f\di\nu =(1\otimes f)(t,x^*)=f(x^*).$$
It follows that $x^*$ is the barycenter of $\lambda$. This completes the proof.
\end{proof}

This lemma may be applied to get the following relationship of two versions of upper envelopes.

\begin{lemma}
\label{l:rezy}
Let $f\in C(K\times B_{E^*},\er)$. Then
\[
\widehat{f}(t,x^*)=(f_t)^*(x^*),\quad (t,x^*)\in K\times B_{E^*},
\]
where $f_t=f(t,\cdot)$.
\end{lemma}

\begin{proof}
Fix $(t,x^*)\in K\times B_{E^*}$. By \cite[Proposition I.5.8]{alfsen} there is $\nu\in M_{(t,x^*)}^+(\K)$ such that
$\widehat{f}(t,x^*)=\int f\di\nu$. By Lemma~\ref{l:repremiry}, $\nu=\ep_t\times \lambda$, where  $r(\lambda)=x^*$. Thus 
\[
\widehat{f}(t,x^*)=\int f\di\nu=\int f_t\di\lambda\le (f_t)^*(x^*),
\]
where the last inequality follows from \cite[Corollary I.3.6]{alfsen}. This completes the proof of `$\le$'.

Conversely, by \cite[Corollary I.3.6]{alfsen} there is $\lambda\in M_1(B_{E^*})$ such that $r(\lambda)=x^*$ and $(f_t)^*(x^*)=\int f_t\di\lambda$. Then $\ep_t\times\lambda\in M_{(t,x^*)}^+(\K)$, and thus 
\[
(f_t)^*(x^*)=\int f_t\di\lambda=\int f\di (\ep_t\times \lambda)\le \widehat{f}(t,x^*),
\]
where the last inequality follows from \cite[Corollary I.5.7]{alfsen}. This completes the proof.
\end{proof}

\begin{lemma}
\label{l:meritelnost}
Let $f\in C(K\times B_{E^*},\er)$ be given. Let $\Sigma$ denote the $\sigma$-algebra generated by Borel rectangles in $K\times B_{E^*}$. Then $\widehat{f}$ is $\Sigma$-measurable.    
\end{lemma}

\begin{proof} To prove that $\widehat{f}$ is $\Sigma$-measurable it is enough to show that it may be uniformly approximated by $\Sigma$-measurable functions. So, fix $\ep>0$. The mapping $t\mapsto f_t=f(t,\cdot)$ is continuous (from $K$ to $C(B_{E^*})$). Hence its range is compact and so there are $t_1,\dots,t_n\in K$ such that $f_{t_1},\dots,f_{t_n}$ form an $\ep$-net of $\{f_t\setsep t\in K\}$.
Therefore we may find a partition $\{A_1,\dots, A_n\}$ of $K$ into nonempty Borel sets such that $\norm{f_t-f_{t_i}}<\ep$ for $t\in A_i$ and $i=1,\dots,n$. It follows that also
$\norm{(f_t)^*-(f_{t_i})^*}_\infty<\ep$ for $t\in A_i$ and $i=1,\dots,n$ (this may be easily deduced using the subadditivity of the upper envelope, see \cite[Proposition I.1.6, (1.7)]{alfsen}).

Since the functions $(f_{t_i})^*$ are weak$^*$ upper semicontinous and thus Borel on $B_{E^*}$, the function 
\[
g(t,x^*)=\sum_{i=1}^n \chi_{A_i}(t)(f_{t_i})^*,\quad (t,x^*)\in K\times B_{E^*},
\]
is $\Sigma$-measurable. Moreover, using Lemma~\ref{l:rezy} we easily deduce that $\norm{\widehat{f}-g}_\infty\le\ep$. This completes the proof.
\end{proof}

Now we are ready to prove the remaining part of Theorem~\ref{t:charakterminim}.

\begin{proof}[Proof of the equivalence $(3)\iff(4)$ from Theorem~\ref{t:charakterminim}]  
$(3)\implies(4)$: Let $g$ be a weak$^*$ continuous convex function on $B_{E^*}$. Then $f=1\otimes g\in -\K$. Hence 
\[
\int_K \left(\int g\di \nu_t\right)\di\abs{\mu}(t)=\int f\di\nu=\int \widehat{f}\di\nu=\int_K 
\left(\int g^*\di\nu_t\right)\di\abs{\mu}(t).
\]
The first equality follows from the formula from Lemma~\ref{L:dezintegrace} together with Proposition~\ref{P:vlastnostih}$(a)$. The second equality follows from $(3)$. To prove the last equality we again use Lemma~\ref{L:dezintegrace}. It may be applied to $\widehat{f}$ due to Lemmata~\ref{l:rezy} and~\ref{l:meritelnost}. 

Taking into account that $g\le g^*$, we deduce that $\int g\di\nu_t=\int g^*\di\nu_t$ 
$\abs{\mu}$-almost everywhere.

$(4)\implies(3)$ Fix $f\in -\K$. The mapping $\f\colon K\to C(B_{E^*})$ defined by $t\mapsto f_t=f(t,\cdot)$ is continuous. Since $\f(K)$ is norm compact set in $C(B_{E^*})$, there exists a countable set $D\subset K$ with $\f(D)$ norm dense in $\f(K)$. It now follows from $(4)$ that there is a $\abs{\mu}$-null set $N\subset K$ such that
$$\forall t\in K\setminus N\; \forall d\in D\colon \int f_d\di\nu_t=\int (f_d)^*\di\nu_t.$$
Fix $t\in K\setminus N$. Then there is a sequence $\{d_n\}$ in $D$ with $\norm{f_{d_n}-f_t}\to 0$. Then also $\norm{(f_{d_n})^*-(f_t)^*}_\infty\to 0$ (this may be easily deduced using the subadditivity of the upper envelope, see \cite[Proposition I.1.6, (1.7)]{alfsen}), and thus
\[
\int f_t\di\nu_t=\lim \int f_{d_n}\di\nu_t=\lim \int(f_{d_n})^*\di\nu_t=\int(f_t)^*\di\nu_t.
\]
Hence 
\[
\int f\di\nu=\int_K \left(\int f_t\di \nu_t\right)\di\abs{\mu}(t)=\int_K \left(\int(f_t)^*\di\nu_t\right)\di\abs{\mu}(t)=\int\widehat f\di\nu.
\]
The first equality follows from Lemma~\ref{L:dezintegrace}. To verify the second one observe that the inner integrals are equal whenever $t\in K\setminus N$ and $\abs{\mu}(N)=0$. The last equality follows by combining Lemma~\ref{l:meritelnost}, Lemma~\ref{L:dezintegrace}  and Lemma~\ref{l:rezy}.

This completes the proof.
\end{proof}

\begin{cor}\label{cor:pseudonesene} Let $\mu\in M(K,E^*)$ be given and let $\nu\in N(\mu)$ be $\prec_{\D}$-minimal. Then the following assertions are valid.
\begin{enumerate}[$(a)$]
    \item $\nu$ is carried by any Baire set containing $K\times \ext B_{E^*}$.
    \item If $B\subset B_{E^*}$ is a Baire set containing $\ext B_{E^*}$ and $(\nu_t)_{t\in K}$ is a disintegration kernel of $\nu$, then $\nu_t$ is carried by $B$ for $\abs{\mu}$-almost all $t\in K$.
\end{enumerate}
\end{cor}

\begin{proof}
    It easily follows from Lemma~\ref{l:repremiry} that  $K\times \ext B_{E^*}$  is the $\K$-Choquet boundary of $K\times B_{E^*}$ (see the definition in \cite[p. 46]{alfsen}). Therefore assertion $(a)$ follows from \cite[Proposition I.5.22]{alfsen}.

    To prove $(b)$, fix $B\subset B_{E^*}$  a Baire set containing $\ext B_{E^*}$. Then $K\times B$ is a Baire set in $K\times B_{E^*}$ containing $K\times \ext B_{E^*}$ and hence by $(a)$ and Lemma~\ref{L:dezintegrace} we get
    $$0=\nu(K\times(B_{E^*}\setminus B))=\int_K \nu_t(B_{E^*}\setminus B)\di\abs{\mu}(t)$$
    and hence the assertion follows.  
\end{proof}

A further consequence is the following theorem.

\begin{thm}\label{t:minimalbodove}
    Let $\nu\in N(\mu)$. Then $\nu$ is $\prec_{\D}$-minimal if and only if it admits a disintegration kernel $(\nu_t)_{t\in K}$ consisting of maximal measures.
\end{thm}

\begin{proof}
    Assume that $(\nu_t)_{t\in K}$ is a disintegration kernel consisting of maximal measures.
    It follows from the Mokobodzki criterion (see Section~\ref{ss:classical}) that condition $(4)$ of Theorem~\ref{t:charakterminim} is satisfied, hence $\nu$ is $\prec_{\D}$-minimal.

    Conversely, assume $\nu$ is $\prec_{\D}$-minimal. Let $(\nu_t)_{t\in K}$ be a disintegration kernel provided by Proposition~\ref{P:vsude}. Fix $g:B_{E^*}\to\er$ weak$^*$ continuous convex. By Theorem~\ref{t:charakterminim} we get that $\int g\di\nu_t=\int g^*\di\nu_t$ $\abs{\mu}$-almost everywhere. Since both $g$ and $g^*$ are bounded Borel functions, the choice of the disintegration kernel yields that the equality holds for all $t\in K$. The Mokobodzki criterion (see Section~\ref{ss:classical}) then implies that each $\nu_t$ is maximal.
\end{proof}

We finish this section by showing that $\prec_\D$-minimal measures are separably determined. To formulate the result properly we recall the notion of a rich family. A~family $\F$ of separable subspaces of a Banach space $E$ is called \emph{rich} if the following two conditions are satisfied:
\begin{itemize}
    \item $\forall F\subset E\mbox{ separable}\,\exists F^\prime\in\F\colon F\subset F^\prime$;
    \item  $\overline{\bigcup_n F_n}\in \F$ whenever $(F_n)$ is an increasing sequence in $\F$.
\end{itemize}

\begin{thm}
    \label{t:sepreduct}
Let $\nu\in N$. Then the following assertions are equivalent.
\begin{enumerate}[$(1)$]
    \item $\nu$ is $\prec_\D$-minimal.
    \item There is a rich family $\F$ of separable subspaces of $E$ such that for each $F\in\F$ the measure $(\id\times R_F)(\nu)$ belongs to $N_F$ and is $\prec_{\D}$ minimal.
    \item There is a cofinal family $\F$ of separable subspaces of $E$ such that for each $F\in\F$ the measure $(\id\times R_F)(\nu)$ belongs to $N_F$ and is $\prec_{\D}$ minimal.
\end{enumerate}    
\end{thm}

\begin{proof}  Implication $(2)\implies (3)$ is trivial.

$(3)\implies(1)$: Assume $(3)$ holds. Let $\nu^\prime\in N$ be such that $\nu^\prime\prec_{\D}\nu$. By Lemma~\ref{L:rovnost tezist}$(b)$ there is some $F_0\subset E$ separable such that for each $F\subset E$ separable containing $F_0$ we have $(\id\times R_F)(\nu^\prime)\in N_F$. 

Fix any $F\in\F$ containing $F_0$. By Proposition~\ref{P:precD-sepred} we deduce $(\id\times R_F)(\nu^\prime)\prec_{\D}(\id\times R_F)(\nu)$. Since both measures $(\id\times R_F)(\nu^\prime)$ and $(\id\times R_F)(\nu)$ belong to $N_F$, we deduce that $(\id\times R_F)(\nu^\prime)=(\id\times R_F)(\nu)$. Since such spaces $F$ form a cofinal family, another use of Proposition~\ref{P:precD-sepred}
shows $\nu^\prime=\nu$ (we are also using the fact that $\prec_{\D}$ is a partial order on $N$ by
Observation~\ref{obs:N}). Thus $\nu$ is $\prec_{\D}$-minimal.

$(1)\implies (2)$: Assume $\nu$ is $\prec_{\D}$-minimal. Let $\mu=T^*\nu$ and let $(\nu_t)_{t\in K}$ be a disintegration kernel for $\nu$. Set
$$\F=\{F\subset E\mbox{ separable}\setsep (\id\times R_F)(\nu)\in N_F\ \&\ (\id\times R_F)(\nu)\mbox{ is $\prec_{\D}$-minimal}\}.$$
We will show that $\F$ is a rich family of separable subspaces of $E$. We start by proving the second property. Assume that $(F_n)$ is an increasing sequence of elements of $\F$ and let $F=\overline{\bigcup_n F_n}$. Clearly $(\id\times R_F)(\nu)\in N_F$. Fix $n\in\en$. By Lemma~\ref{L:dezintegrace kvocient} $(R_{F_n}(\nu_t))_{t\in K}$ is a disintegration kernel of $(\id\times R_{F_n})(\nu)$, so by Proposition~\ref{P:minD-separable} we deduce that
$R_{F_n}(\nu_t)$ is a maximal measure on $B_{F_n^*}$ for $\abs{\mu}$-almost all $t\in K$ (note that our assumptions imply $\pi_1((\id\times R_{F_n})(\nu))=\abs{\mu}$). For each $t\in K$ the measure 
$R_F(\nu_t)$ is the inverse limit of the sequence $(R_{F_n}(\nu))_n$ and hence for $\abs{\mu}$-almost all $t\in K$ it is a maximal measure (by \cite[Theorem 12.31]{lmns}). Hence $(\id\times R_F)(\nu)$ is $\prec_{\D}$-minimal by Proposition~\ref{P:minD-separable}. Thus $F\in\F$ and the second property holds.

It remains to prove the cofinality of $\F$. To this end denote by $\Co(B_{E^*})$ the convex cone of all weak$^*$ continuous convex functions on $B_{E^*}$. The proof will proceed in several steps.
\smallskip

{\tt Step 1:} For any $g\in \Co(B_{E^*})$ there exists a countable set $C_g\subset -\Co(B_{E^*})$ and a set $K_g\subset K$ of full measure $\abs{\mu}$ such that  
\begin{itemize}
    \item for any $k\in C_g$ it holds $k\ge g$;
    \item for any $\ep>0$ and $t\in K_g$  there exists $h\in C_g$ with $\int h\di\nu_t\le\int g\di \nu_t+\ep$.
\end{itemize}
\smallskip

Indeed, given $g\in \Co(B_{E^*})$, the function $(1\otimes g)(t,x^*)=g(x^*)$ belongs to $-\K$, and hence $\int (1\otimes g)\di\nu=\int(\widehat{1\otimes g})\di\nu$ by Theorem~\ref{t:charakterminim}. By the monotone convergence theorem for nets we deduce 
$$\int  (1\otimes g)\di\nu=\inf \left\{\int k\di\nu\setsep k\in \K, k\ge 1\otimes g\right\}.$$
Hence there exists a nonincreasing sequence $\{k_n\}$ of functions from $\K$ such that $k_n\ge 1\otimes g$ for each $n$ and
$\int (1\otimes g)\di\nu=\inf_{n\in\en} \int k_n\di\nu$.
For each $n\in\en$ we consider a countable set $K_n\subset K$ such that
$\{(k_n)_t\setsep t\in K_n\}$ is norm dense in $\{(k_n)_t\setsep t\in K\}$ (note that $(k_n)_t=k_n(t,\cdot)$, as above).
Let 
\[
C_g=\bigcup_{n\in\en}\{(k_n)_t\setsep t\in K_n\}.
\]
Then $C_g\subset -\Co(B_{E^*})$ and every function from $C_g$ is greater or equal than $g$. 

Further, $k=\inf_{n\in\en} k_n$ satisfies $\int(1\otimes g)\di\nu=\int k\di\nu$, and hence $\int g\di\nu_t=\int k_t\di\nu_t$ for $\abs{\mu}$-almost all $t\in K$. Let us denote the respective set of full measure $\abs{\mu}$ by $K_g$. 

Then $K_g$ and $C_g$ satisfy the required properties. Indeed, let $\ep>0$ be given. For $t\in K_g$ we have $\int g\di\nu_t=\int k_t\di\nu_t=\inf_{n\in\en} \int (k_n)_t\di\nu_t$. Hence there exists $n\in\en$ with $\int(k_n)_t\di\nu_t\le \int g\di\nu_t+\frac{\ep}{3}$. Let $h\in C_g$ be chosen such that $\norm{h-(k_n)_t}<\frac{\ep}{3}$. Then 
\[
\int h\di\nu_t\le \int(k_n)_t\di\nu_t+\tfrac{\ep}{3}\le \int g\di\nu_t+\tfrac{2\ep}{3}.
\]
Hence the family $C_g$ along with the set $K_g$ have the desired properties.

\smallskip
{\tt Step 2:} For any norm-separable $S\subset \Co(B_{E^*})$ there exists a countable set $C_S\subset -\Co(B_{E^*})$ and a set $K_{S}\subset K$ of full measure $\abs{\mu}$ such that for 
for any $\ep>0$, $g\in S$ and $t\in K_{S}$  there exists $h\in C_S$ with $h+\ep\ge g$ and $\int h\di\nu_t\le\int g\di \nu_t+\ep$.

\smallskip

Let $A\subset S$ be a countable norm-dense set. It is enough to set 
\[
C_S=\bigcup_{g\in A} C_g\quad\text{and}\quad K_S=\bigcap_{g\in A} K_g, 
\]
where $C_g$ and $K_g$ are constructed for the function $g\in A$ as in the the first step.

\smallskip

{\tt Step 3:} Fix $F_0\subset E$ separable. We construct inductively norm-separable sets $S_0\subset S_1\subset S_2\subset\cdots\subset \Co(B_{E^*})$, sets $K=K_0\supset K_1\supset K_2\supset K_3\supset\cdots$ of full measure  $\abs{\mu}$ and separable spaces $F_0\subset F_1\subset F_2\subset\cdots E$ as follows.

In the first step of the construction, we set $K_0=K$ and 
\[
S_0=\{g\circ R_{F_0}\setsep g\in\Co(B_{F_0^*})\}.
\]

Assume that $S_{n-1}$, $K_{n-1}$ and $F_{n-1}$ have been constructed.
We apply Step 2 for $S_{n-1}$ to find a  countable set $C\subset -\Co(B_{E^*})$ along with the set $K_{n}\subset K$ (without loss of generality $K_n\subset K_{n-1}$) with the properties described in Step 2. Let $H\subset E$ be a separable space such that $C$ can be factorized via $H$ in the sense of Lemma~\ref{l:separ-faktor}$(b)$. Let $F_n=\ov{\span(F_{n-1}\cup H)}$ and 
\[
S_n=\{g\circ R_{F_n}\setsep g\in\Co(B_{F_n^*})\}\cup S_{n-1}.
\]
Then all elements from $S_{n}\cup C$ can be factorized via $F_n$. 
This finishes the inductive construction.

To conclude the proof we set $F=\ov{\bigcup_{n=0}^\infty F_n}$. 
Now $(R_F(\nu_t))_{t\in K}$ is a disintegration kernel for $(\id\times R_F)(\nu)$ (by Lemma~\ref{L:dezintegrace kvocient}). We want to check that $R_F(\nu_t)$ is maximal for each $t\in\bigcap_{n=0}^\infty K_n$. 

So, fix such $t$, $g=g_1\vee\dots\vee g_k$ where $g_1,\dots,g_k$ are real-valued weak$^*$ continuous affine functions on $B_{F^*}$ and $\ep>0$.
Then there are $x_1,\dots,x_k\in F$ and $c_1,\dots,c_k\in\er$ such that $g_j(x^*)=c_j+\Re x^*(x_j)$ for $x^*\in B_{F^*}$ and $j=1,\dots,k$. By the choice of $F$ there is $n\in \en$ and elements
$x^\prime_1,\dots,x^\prime_k\in F_n$ with $\norm{x_j-x^\prime_j}<\ep$ for $j=1,\dots,k$. Set
$$g^\prime(x^*)=\min_{1\le j\le k} (c_j+\Re x^*(x^\prime_j)),\quad x^*\in B_{F^*}.$$
Then $\norm{g-g^\prime}_\infty<\ep$ and $g^\prime\circ R_F\in S_n$.
By the inductive construction there is $h\in S_{n+1}$ such that 
 $h+\ep\ge g^\prime\circ R_F$ and $\int h\di\nu_t\le\int g^\prime\circ R_F\di \nu_t+\ep$. Clearly $h=h'\circ R_F$ for some $h'\in-\Co(B_{F^*})$.

Then $h'+\ep$ is in $-\Co(B_{F^*})$ and
\[
(h'+\ep)\circ R_F=h+\ep\ge g'\ge g\circ R_F-\ep. 
\]
Hence $g\le h'+2\ep$ on $B_{F^*}$, and thus
\[
\begin{aligned}
\int g^*\di R_F(\nu_t)&\le \int(h'+2\ep)\di R_F(\nu_t)=\int (h+2\ep)\di\nu_t\le \int g'\di\nu_t+3\ep\\
&\le \int (g\circ R_F)\di\nu_t+4\ep=\int g\di R_F(\nu_t)+4\ep.
\end{aligned}
\]
Since $\varepsilon>0$ is arbitrary, we get $\int g^*\di R_F(\nu_t)=\int g\di R_F(\nu_t)$. This implies that $\nu_t$ is maximal (by a version of Mokobodzki test, cf. \cite[Theorem 3.58$(i)\iff(ii)$]{lmns}). Hence $R_F(\nu_t)$ is maximal for $\abs{\mu}$-almost all $t$, which gives that $(\id\times R_F)(\nu)$ is minimal (by Proposition~\ref{P:minD-separable}).
\end{proof}

\subsection{Uniqueness of $\prec_{\D}$-minimal measures}

It is easy to show that for any $\mu\in M(K,E^*)$ there is a $\prec_{\D}$-minimal measure in $N(\mu)$.
In this section we address the question of uniqueness of such a measure. To this end we will use the notion of simplexoid introduced in \cite{phelps-complex}. Recall that a convex set $X$ is called \emph{simplexoid} if every  proper face of $X$ is a simplex. It is a geometrical notion, but in case of dual unit balls it may be characterized using representing mesures.

\begin{fact}
    Let $E$ be a Banach space. Then $B_{E^*}$ is a simplexoid if and only if for each $x^*\in S_{E^*}$ there is a unique maximal probability on $B_{E^*}$ with barycenter $x^*$.
\end{fact}

\begin{proof}
The assertion follows from the proof of \cite[Theorem 3.11]{fuhr-phelps}.    
\end{proof}

The following theorem characterizes uniqueness of $\prec_{\D}$ minimality measures. We note that implication $(2)\implies(1)$ is essentially trivial, so the key result is implication $(1)\implies(2)$. We also point out that condition $(1)$ does not depend on $K$, so the uniqueness depends just on the target space $E$.

\begin{thm}
\label{t:simplexoidminimal}
The following assertions are equivalent.
\begin{enumerate}[$(1)$]
    \item $B_{E^*}$ is a simplexoid.
    \item For each $\mu\in M(K,E^*)$  there exists a unique $\prec_\D$-minimal measure $\nu\in N(\mu)$.    
\end{enumerate}
\end{thm}

\begin{proof}
$(2)\implies(1)$: Assume that $B_{E^*}$ is not a simplexoid. Then there is some $x^*\in S_{E^*}$ and two distinct maximal measures $\omega_1,\omega_2$ with the barycenter $x^*$. Fix $t\in K$. Consider two measures
$$\nu_1=\ep_t\times\omega_1,\quad \nu_2=\ep_t\times\omega_2.$$
Then $T^*\nu_1=T^*\nu_2=\ep_t\otimes x^*$ and $\nu_1,\nu_2\in N(\ep_t\otimes x^*)$. Moreover,
both $\nu_1$ and $\nu_2$ are $\prec_{\D}$-minimal (for example) by Theorem~\ref{t:minimalbodove}.

$(1)\implies(2)$: Assume $B_{E^*}$ is a simplexoid. Fix $\mu\in M(K,E^*)$. Choose the assignment of disintegration kernels provided by Proposition~\ref{P:vsude}. Let $\nu_1,\nu_2\in N(\mu)$ be a pair of $\prec_{\D}$ minimal measures. As in Theorem~\ref{t:minimalbodove} we see that $\nu_{1,t}$ and $\nu_{2,t}$ are maximal for $t\in K$. Moreover, let $\nu_0=W\mu$. 

Fix $j\in\{1,2\}$. By Lemma~\ref{L:precD}$(c)$ we know that $\nu_j\prec_{\D}\nu_0$, hence $\nu_{0,t}\prec\nu_{1,t}$ for each $t\in K$ (as in Theorem~\ref{t:precDbodove}). Thus $r(\nu_{0,t})=r(\nu_{j,t})$ for each $t\in K$. We deduce that $r(\nu_{1,t})=r(\nu_{2,t})$  for $t\in K$. Since $r(\nu_{1,t})\in S_{E^*}$ $\abs{\mu}$-almost everywhere (by Proposition~\ref{P:vlastnostih}$(b)$), the assumption that $B_{E^*}$ is a simplexoid yields that $\nu_{1,t}=\nu_{2,t}$ $\abs{\mu}$-almost everywhere, hence $\nu_1=\nu_2$.
\end{proof}

\section{Overview of the results}

In this final section we present a brief overview of the results from this paper and of the related context. 

\begin{itemize}
    \item The continuous functionals on $C(K,E)$ are in one-to-one isometric correspondence with $E^*$-valued regular Borel measures on $K$. It is the content of Singer's representation theorem (an easy proof is given in \cite{hensgen} and recalled in Section~\ref{s:dualita}).
    \item Since the canonical inclusion $T:C(K,E)\to C(K\times B_{E^*})$ is an isometry, 
   any $\mu\in M(K,E^*)$ admits some $\nu\in M(K\times B_{E^*})$ with $\norm{\nu}=\norm{\mu}$ such that $T^*\nu=\mu$. This is just a consequence of Hahn-Banach extension theorem and Riesz representation theorem. The vector measure $\mu$ may be computed from $\nu$ by the Hustad formula \eqref{eq:hustad}.
   \item The measure $\nu$ in the previous item may be chosen positive. We denoted the set of such measures $N(\mu)$, i.e.,
   $$N(\mu)=\{\nu\in M_+(K\times B_{E^*})\setsep \norm{\nu}=\norm{\mu}\ \&\ T^*\nu=\mu\}.$$
   Moreover, there is a canonical selection operator $W$ from the assignment $\mu\mapsto N(\mu)$.
   This operator was constructed in \cite{batty-vector}, we present an alternative approach using the method of disintegration, see Proposition~\ref{P:KT*nu konstrukce} and Lemma~\ref{L:KT*nu dukaz}.
   \item If $E^*$ is strictly convex, then $N(\mu)$ is a singleton for each $\mu\in M(K,E^*)$. It is established in Theorem~\ref{t:battymain}.
   \item If $E^*$ is not strictly convex, then $N(\mu)$ is a larger set (at least for some $\mu$). There is a natural partial order $\prec_{\D}$ on $N(\mu)$. In this order $W\mu$ is the unique maximal measure (see Proposition~\ref{P:precD na sfere}). Further, minimal measures exist and are pseudosupported by $K\times \ext B_{E^*}$ (see Proposition~\ref{P:minD-separable} and Corollary~\ref{cor:pseudonesene}). Minimal measures are unique if and only if $B_{E^*}$ is simplexoid (see Theorem~\ref{t:simplexoidminimal}).
   \item The order $\prec_{\D}$ is closely related to the Choquet order on $B_{E^*}$ (see Corollary~\ref{cor:precD-separable} and Theorem~\ref{t:precDbodove}) and $\prec_{\D}$-minimal measures are closely related to maximal measures on $B_{E^*}$ (see Proposition~\ref{P:minD-separable} and Theorem~\ref{t:minimalbodove}).
\end{itemize}

\bibliographystyle{acm}
\bibliography{vector}

\begin{thebibliography}{10}

\bibitem{alfsen}
{\sc Alfsen, E.}
\newblock {\em Compact convex sets and boundary integrals}.
\newblock Springer-Verlag, New York, 1971.
\newblock Ergebnisse der Mathematik und ihrer Grenzgebiete, Band 57.

\bibitem{hypolinear}
{\sc Anger, B., and Lembcke, J.}
\newblock Hahn-{B}anach type theorems for hypolinear functionals.
\newblock {\em Math. Ann. 209\/} (1974), 127--151.

\bibitem{batty-vector}
{\sc Batty, C. J.~K.}
\newblock Vector-valued {C}hoquet theory and transference of boundary measures.
\newblock {\em Proc. London Math. Soc. (3) 60}, 3 (1990), 530--548.

\bibitem{diesteluhl}
{\sc Diestel, J., and Uhl, Jr., J.~J.}
\newblock {\em Vector measures}.
\newblock American Mathematical Society, Providence, R.I., 1977.
\newblock With a foreword by B. J. Pettis, Mathematical Surveys, No. 15.

\bibitem{fremlin3}
{\sc Fremlin, D.~H.}
\newblock {\em Measure theory. {V}ol. 3}.
\newblock Torres Fremlin, Colchester, 2004.
\newblock Measure algebras, Corrected second printing of the 2002 original.

\bibitem{fremlin4}
{\sc Fremlin, D.~H.}
\newblock {\em Measure theory. {V}ol. 4}.
\newblock Torres Fremlin, Colchester, 2006.
\newblock Topological measure spaces. Part I, II, Corrected second printing of
  the 2003 original.

\bibitem{fuhr-phelps}
{\sc Fuhr, R., and Phelps, R.~R.}
\newblock Uniqueness of complex representing measures on the {C}hoquet
  boundary.
\newblock {\em J. Functional Analysis 14\/} (1973), 1--27.

\bibitem{hensgen}
{\sc Hensgen, W.}
\newblock A simple proof of {S}inger's representation theorem.
\newblock {\em Proc. Amer. Math. Soc. 124}, 10 (1996), 3211--3212.

\bibitem{hirsberg72}
{\sc Hirsberg, B.}
\newblock Repr\'{e}sentations int\'{e}grales des formes lin\'{e}aires
  complexes.
\newblock {\em C. R. Acad. Sci. Paris S\'{e}r. A-B 274\/} (1972), A1222--A1224.

\bibitem{hustad71}
{\sc Hustad, O.}
\newblock A norm preserving complex {C}hoquet theorem.
\newblock {\em Math. Scand. 29\/} (1971), 272--278.

\bibitem{lmns}
{\sc Luke{\v{s}}, J., Mal{\'y}, J., Netuka, I., and Spurn{\'y}, J.}
\newblock {\em Integral representation theory}, vol.~35 of {\em de Gruyter
  Studies in Mathematics}.
\newblock Walter de Gruyter \& Co., Berlin, 2010.
\newblock Applications to convexity, Banach spaces and potential theory.

\bibitem{michael}
{\sc Michael, E.}
\newblock Continuous selections. {I}.
\newblock {\em Ann. of Math. (2) 63\/} (1956), 361--382.

\bibitem{phelps-complex}
{\sc Phelps, R.~R.}
\newblock The {C}hoquet representation in the complex case.
\newblock {\em Bull. Amer. Math. Soc. 83}, 3 (1977), 299--312.

\bibitem{phelps-choquet}
{\sc Phelps, R.~R.}
\newblock {\em Lectures on {C}hoquet's theorem}, second~ed., vol.~1757 of {\em
  Lecture Notes in Mathematics}.
\newblock Springer-Verlag, Berlin, 2001.

\bibitem{roth-london}
{\sc Roth, W.}
\newblock A new concept for a {C}hoquet ordering.
\newblock {\em J. London Math. Soc. (2) 34}, 1 (1986), 81--96.

\bibitem{roth-convex}
{\sc Roth, W.}
\newblock Choquet theory for vector-valued functions on a locally compact
  space.
\newblock {\em J. Convex Anal. 21}, 4 (2014), 1141--1164.

\bibitem{roth-kniha}
{\sc Roth, W.}
\newblock {\em Integral representation---{C}hoquet theory for linear operators
  on function spaces}, vol.~74 of {\em De Gruyter Expositions in Mathematics}.
\newblock De Gruyter, Berlin, [2023] \copyright 2023.

\bibitem{saab-aeq}
{\sc Saab, P.}
\newblock The {C}hoquet integral representation in the affine vector-valued
  case.
\newblock {\em Aequationes Math. 20}, 2-3 (1980), 252--262.

\bibitem{saab-canad}
{\sc Saab, P.}
\newblock Integral representation by boundary vector measures.
\newblock {\em Canad. Math. Bull. 25}, 2 (1982), 164--168.

\bibitem{saab-tal}
{\sc Saab, P., and Talagrand, M.}
\newblock A {C}hoquet theorem for general subspaces of vector-valued functions.
\newblock {\em Math. Proc. Cambridge Philos. Soc. 98}, 2 (1985), 323--326.

\end{thebibliography}

\end{document}